\documentclass{amsart}

\usepackage{amsmath,amssymb,amsthm,mathrsfs}
\theoremstyle{plain}
\numberwithin{equation}{section}
\theoremstyle{plain}
\newtheorem{theorem}{Theorem}[section]
\newtheorem{corollary}[theorem]{Corollary}
\newtheorem{lemma}[theorem]{Lemma}
\newtheorem{proposition}[theorem]{Proposition}
\theoremstyle{definition}

\newtheorem{example}[theorem]{Example}
\newtheorem{remark}[theorem]{Remark}

\newcommand{\ands}{\quad\mbox{and}\quad}
\newcommand{\kernel}{{\mathrm{Ker}\, }}
\newcommand{\degr}{{\mathrm{deg}\, }}
\newcommand{\Dom}{{\mathrm{Dom}}}
\newcommand{\codim}{{\mathrm{codim}\, }}
\newcommand{\Index}{{\mathrm{Index}}}
\newcommand{\Rat}{{\mathrm{Rat}}}
\newcommand{\Ran}{{\mathrm{Ran}\, }}
\newcommand{\diag}{{\mathrm{Diag}}}

\newcommand{\cP}{{\mathcal P}}

\newcommand{\BC}{{\mathbb C}}
\newcommand{\BT}{{\mathbb T}}

\newcommand{\BD}{{\mathbb D}}
\newcommand{\BN}{{\mathbb N}}
\newcommand{\BP}{{\mathbb P}}

\newcommand{\wtil}[1]{{\widetilde{#1}}}
\newcommand{\what}[1]{{\widehat{#1}}}

\newcommand{\al}{\alpha}

\newcommand{\La}{\Lambda}

\newcommand{\Up}{\Upsilon}

\newcommand{\om}{\omega}\newcommand{\Om}{\Omega}
\newcommand{\ov}[1]{{\overline{#1}}}
\newcommand{\un}[1]{{\underline{#1}}}

\newcommand{\sbm}[1]{\left(\begin{smallmatrix} #1\end{smallmatrix}\right)}
\newcommand{\mat}[1]{\ensuremath{\begin{pmatrix} #1 \end{pmatrix}}}

\begin{document}


\title{A Toeplitz-like operator with rational matrix symbol having poles on the unit circle: Fredholm properties}

\author[G.J. Groenewald]{G.J. Groenewald}
\address{School of Mathematical and Statistical Sciences,
North-West University,
Research Focus Area: Pure and Applied Analytics,
Private Bag X6001,
Potchefstroom 2520,
South Africa.}
\email{Gilbert.Groenewald@nwu.ac.za}

\author[S. ter Horst]{S. ter Horst}
\address{School of Mathematical and Statistical Sciences,
North-West University,
Research Focus Area: Pure and Applied Analytics,
Private Bag X6001,
Potchefstroom 2520,
South Africa and DSI-NRF Centre of Excellence in Mathematical and Statistical Sciences (CoE-MaSS)}
\email{Sanne.TerHorst@nwu.ac.za}

\author[J. Jaftha]{J. Jaftha}
\address{Numeracy Center, University of Cape Town, Rondebosch 7701; Cape Town; South Africa}
\email{Jacob.Jaftha@uct.ac.za}

\author[A.C.M. Ran]{A.C.M. Ran}
\address{Department of Mathematics, Faculty of Science, VU Amsterdam, De Boelelaan 1111, 1081 HV Amsterdam, The Netherlands and Research Focus Area: Pure and Applied Analytics, North-West~University, Potchefstroom, South Africa}
\email{a.c.m.ran@vu.nl}

\thanks{This work is based on the research supported in part by the National Research Foundation of South Africa (Grant Numbers 118513 and 127364).}



\begin{abstract}
This paper concerns the analysis of an unbounded Toeplitz-like operator generated by a rational matrix function having poles on the unit circle $\BT$. It extends the analysis of such operators generated by scalar rational functions with poles on $\BT$ found in \cite{GtHJR1,GtHJR2,GtHJR3}. A Wiener-Hopf type factorization of rational matrix functions with poles and zeroes on $\BT$ is proved and then used to analyze the Fredholm properties of such Toeplitz-like operators. A formula for the index, based on the factorization, is given. Furthermore, it is shown that the determinant of the matrix function having no zeroes on $\BT$ is not sufficient for the Toeplitz-like operator to be Fredholm, in contrast to the classical case.
\end{abstract}


\maketitle

\section{Introduction}

This paper is a continuation of \cite{GtHJR1,GtHJR2,GtHJR3}, where Toeplitz-like operators with rational symbols with poles on the unit circle $\BT$ were studied. Whilst the aim of \cite{GtHJR1,GtHJR2,GtHJR3} was to analyze such Toeplitz-like operators with scalar symbols, in this paper we will focus on such Toeplitz-like operators with matrix symbol.

First we introduce some notation. For positive integers $m$ and $n$, let $\Rat^{m\times n}$ denote the space of $m \times n$ rational matrix functions, abbreviated to $\Rat^{m}$ if $n=1$. We write $\Rat^{m\times  n}(\BT)$ for the functions in $\Rat^{m\times n}$ with poles only on $\BT$ and with $\Rat_0^{m\times  n}(\BT)$ we indicate the functions in $\Rat^{m\times n}$ that are strictly proper, that is, whose limit at $\infty$ exists and is equal to the zero-matrix. Also here, if $n=1$ we write $\Rat^{m}(\BT)$ and $\Rat_0^{m}(\BT)$ instead of $\Rat^{m\times 1}(\BT)$ and $\Rat_0^{m\times 1}(\BT)$, respectively. In the scalar case, i.e., $m=n=1$, we simply write $\Rat$, $\Rat(\BT)$ and $\Rat_0(\BT)$, as was done in \cite{GtHJR1,GtHJR2,GtHJR3}.

A rational matrix function has a pole at $z$ if any of its entries has a pole at $z$. A zero of a square rational matrix function $\Om$ is a pole of its inverse $\Om^{-1}(z)=\Om(z)^{-1}$, see for example \cite{DS74}. In the case of a square matrix polynomial $P$, the zeroes coincide with the zeroes of the polynomial $\det P(z)$, but for square rational matrix functions, zeroes can also occur at points where the determinant is nonzero.

The space of $m\times n$ matrix polynomials will be denoted by $\cP^{m\times n}$, abbreviated to $\cP^m$ if $n=1$ and to $\cP$ if $m=n=1$. For all positive integers $k$ we write $\cP^{m\times n}_k$, $\cP^{m}_k$ and $\cP_k$ for the polynomials in $\cP^{m\times n}$, $\cP^{m}$ and $\cP$ of degree at most $k$. By $L_m^p$ and $H_m^p$ we shall mean vector-functions with $m$-components who are all in $L^p$ and $H^p$, respectively.

We can now define our Toeplitz-like operators. Let $\Om \in\Rat^{m\times  m}$ with possibly poles on $\BT$ and $\det\Om(z)\not \equiv 0$. Define
$T_\Om \left (H^p_m \rightarrow H^p_m \right ), 1 < p < \infty,$ by
\begin{equation}\label{TOm}
\begin{aligned}
\Dom(T_\Om)=\left\{\begin{array}{ll}
f\in H^p_m :&
\begin{array}{l}
\Om f = h + r \textrm{ where } h\in L^p_m(\BT),\\
\ands r \in\Rat_0^{m}(\BT) \\
\end{array}\\
\end{array}
\right \}, \\
T_\Om f = \BP h \textrm{ with } \BP \textrm{ the Riesz projection of } L^p_m(\BT) \textrm{ onto }H^p_m.
\end{aligned}
\end{equation}
By the Riesz projection, $\BP$, we mean the projection of $L^p_m$ onto $H^p_m$ due to M. Riesz, see for example pages 149 - 153 in \cite{H62}, in contrast to the Riesz projection in spectral operator theory, due to F. Riesz, see for example pages 9 - 13 in \cite{GGK92a}. For simplicity we will consider the square case only, but many of the results in this chapter extend to the non-square case, i.e., $m \not = n$.

The aim of this paper is the determination of Fredholm properties of the Toeplitz-like operator $T_\Om$ defined above. For the scalar case, the Fredholm-properties of Toeplitz-like operators were studied in \cite{GtHJR1}. For the classical case, i.e., without poles on the unit circle, the Fredholm-properties appear in Chapters XXIII and XXIV of \cite{GGK2}. In that case, the block Toeplitz operator is Fredholm exactly when the determinant of the symbol has no zeroes on $\BT$ (Theorem XXIII.4.3 in \cite{GGK2}). As we will see later, cf. Section \ref{S:FredholmM}, this is not the case when poles on the unit circle are allowed, due to possible pole-zero cancellation.

An essential ingredient in the analysis of Fredholm properties of Toeplitz operators with matrix symbols is played by Wiener-Hopf factorization; cf.,  Theorem XXIV.3.1 of \cite{GGK2}. This allows one to determine invertibility conditions and Fredholm-properties of the block Toeplitz operator; cf., Theorem XXIV.4.1 and Theorem XXIV.4.2 of \cite{GGK2}. Here too we base our analysis on a Wiener-Hopf-type factorization.\smallskip

\paragraph{\bf Main Results} In our first main result, using an adaptation of the construction in \cite{CG} we prove a Wiener-Hopf type factorization for a rational matrix function with poles on the unit circle. We call an $m\times m$ rational matrix function a plus function if it has no poles in the closed unit disk $\ov{\BD}$ and we call it a minus function if it has no poles outside the open unit disk $\BD$, with the point at infinity included.

\begin{theorem}\label{T:fact}
Let $\Om\in\Rat^{m\times m}$ with $\det \Om \not\equiv 0$. Then
\begin{equation}\label{WHfact1}
\Om(z) = z^{-k}\Om_-(z) \Om_\circ(z) P_0(z) \Om_+(z),
\end{equation}
for some $k\geq0$, $\Om_+,\Om_\circ,\Om_-\in\Rat^{m\times m}$ and $P_0\in\cP^{m\times m}$ such that $\Om_-$ and $(\Om_-)^{-1}$ are minus functions, $\Om_+$  and $(\Om_+)^{-1}$ are plus functions, $\Om_\circ = \diag_{j=1}^m (\phi_j)$ with each $\phi_j\in\Rat(\BT)$ having zeroes and poles only on $\BT$, and $P_0$ is a lower triangular matrix polynomial with $\det (P_0(z)) = z^N$ for some integer $N \geq 0$.
\end{theorem}

Note that, the diagonal entries of $P_0$ must be of the form $z^{n_1},\ldots,z^{n_m}$. However, since $P_0$ is assumed to be lower triangular and not diagonal, it is not possible to order them (increasing or decreasing in power of $z$) without disrupting the lower triangular structure. This is in sharp contrast to the classical Wiener-Hopf factorization result where the entries on the diagonal can be ordered to have increasing degrees and subject to this ordering become unique, see for example \cite{CG,GK58}.

In the case studied here, with poles on $\BT$ allowed, we do not find a very satisfactory uniqueness claim. As in the classical case, the plus function $\Omega_+$ and $\Omega_-$ are, in general, not unique. However, also the lower triangular matrix polynomial $P_0$ appears to be far from unique (as shown in Examples \ref{E:notUnique1} and \ref{E:notUnique2} below). Subject to the restriction of a special form of the factorization \eqref{WHfact1}, which can always be obtained, it is possible to prove that the factor $\Omega_\circ$ is unique.

\begin{theorem}\label{T:factUnique}
Let $\Om\in\Rat^{m\times m}$ with $\det \Om \not\equiv 0$. Let $q(z)$ be the least common multiple of all denominators of the entries of $\Omega$, and write $q(z)=q_-(z)q_\circ(z)q_+(z)$, where $q_+$ has zeroes only outside $\overline{\mathbb{D}}$, $q_-$ has zeroes only inside $\mathbb{D}$, and $q_\circ$ has zeroes only on $\mathbb{T}$. Then $\Om$ admits a factorization \eqref{WHfact1} as in Theorem \ref{T:fact} with the additional properties that
\begin{equation}\label{P+DcircP-}
\begin{aligned}
&P_+(z):=q_+(z)\Om_+(z),\quad D_\circ(z):=q_\circ(z)\Om_\circ(z),\\
&\quad\qquad P_-(\mbox{$\frac{1}{z}$}):=z^{-\degr q_-} q_-(z)\Om_-(z),
\end{aligned}
\end{equation}
define matrix polynomials $P_+$ and $P_-$ with no roots inside $\overline{\BD}$, $D_\circ$ is a diagonal polynomial, $D_\circ(z)=\diag_{j=1}^m(p^\circ_j(z))$ with $p^\circ_{j+1}$ dividing $p^\circ_{j}$ for $j=1,\ldots,m-1$, $k$ is the smallest number so that $P_0$ is still a polynomial, i.e., if $k>0$ then $P_0(0)\neq 0$, and the off-diagonal entries in $P_0$ have a lower degree than the diagonal entry in the same row. Among all factorizations \eqref{WHfact1} of $\Om$ satisfying these additional conditions, $\Om_\circ$ is uniquely determined.
\end{theorem}

In fact, as observed in Corollary \ref{C:SFP1} below, the diagonal matrix $D_\circ$ correspond to what we define in Section \ref{S:PolyFact} as the Smith form of the polynomial $P_1(z):=q(z)\Om(z)$ with respect to $\BT$, and can be directly determined by the classical Smith form of $P_1$.

In Example \ref{E:notUnique2} below, for the case of $2 \times 2$ matrix function, subject to an additional constraint it is possible to diagonalize the factor $\Omega_\circ(z)P_0(z)$ in \ref{WHfact1}. Whether it is possible to diagonalize $\Omega_\circ(z)P_0(z)$ in general remains unclear, we have proof nor counterexample, even for the $2 \times 2$ case. However, the arguments used in the proofs of the classical cases, without poles on $\BT$, do not appear to generalize easily to the case considered in the present paper.

As in the case of scalar rational functions, the Wiener-Hopf type factorization of Theorem \ref{T:fact} allows us to factorize the Toeplitz-like operator along the matrix function factorization.

\begin{theorem}\label{T:Fact_Toeplitz}
Let $\Om(z) = z^{-k}\Om_-(z) \Om_\circ(z) P_0(z) \Om_+(z)$ be a Wiener-Hopf type factorization of $\Om$ as in Theorem \ref{T:fact}.
Then
\[
T_\Om = T_{\Om_-} T_{z^{-k}}T_{\Om_\circ}T_{P_0} T_{\Om_+}.
\]
\end{theorem}

Our next, and final, main result uses the factorization of $T_\Om$ to characterize the Fredholmness of $T_\Om$ and determine the Fredholm index of $T_\Om$ in terms of properties of the functions of the Wiener-Hopf type factorization of $\Om$ in case $T_\Om$ is Fredholm.

\begin{theorem}\label{T:Fredholm}
Let $\Om(z) = z^{-k}\Om_-(z) \Om_\circ(z) P_0(z) \Om_+(z)$ be the Wiener-Hopf type factorization of $\Om$ as in Theorem \ref{T:fact} with $\Om_\circ(z)=\diag_{j=1}^n (s_j(z)/q_j(z))$ with $s_j,q_j\in\cP$ co-prime with roots only on $\BT$ and $z^{n_j}$ the $j$-th diagonal entry of $P_0$, for $j=1,\ldots, m$. Then $T_\Om$ is Fredholm if and only if $T_{\Om_\circ}$ is Fredholm, which happens exactly when all $s_1,\ldots,s_m$ are constant. In case $T_\Om$ is Fredholm we have
\begin{align*}
\Index (T_\Om) & = mk + \Index(T_{\Om_\circ}) + \Index(T_{P_0})\\
& = mk + \sum_{j=1}^m \degr q_j  - \sum_{j=1}^m n_j.
\end{align*}
\end{theorem}

\paragraph{\bf Comparison with the literature.}
There are not many known results concerning unbounded (or closed) Toeplitz operators with matrix symbols. We use a factorization based on the Smith form \index{Smith form} for matrix  polynomials (Theorem 1 and Theorem 2 of Gantmacher, chapter VI \cite{G59}) which resembles the Wiener-Hopf factorization for rational matrix functions without poles on the unit circle appearing in Theorem 2.1 of \cite{CG}. The proof of Theorem 2.1 of \cite{CG} uses the Smith form as a starting point.

Factorization of matrix functions, however, has a long tradition. Wiener-Hopf type factorizations of matrix functions relative to a given contour (with certain properties, but not necessarily equal to $\BT$) on which they can have no poles, appeared, for example, in Gohberg and Krein \cite{GK58} and Clancey and Gohberg \cite{CG}; factorization of matrix functions as solutions to barrier problems in complex function theory appear as early as 1908 in Plemelj \cite{P08},  whereas factorization of rational matrix functions relative to the unit circle appears as spectral factorization in electrical engineering in Belevitch \cite{B68} and Youla \cite{Y61}. In all of these cases, however, there is the restriction that the (matrix) functions can have no poles on the contour.

The proof of the Wiener-Hopf type factorization for a rational matrix function with poles on the unit circle of Theorem \ref{T:fact} is based on the approach found in \cite{CG}. However, there is a slight oversight in the proof of Theorem 2.1 in Chapter 1 of \cite{CG}. An application of Lemma \ref{L:fact_3} below would eliminate this. It is not true that $LD(z)P(z) = D(z)LP(z)$ where $L$ is a lower triangular elementary matrix with ones on the main diagonal and only one row of nonzero entries off the main diagonal, $D(z) = \diag (z^{n_j})$ is a diagonal matrix with $n_1 \geq n_2 \geq \ldots \geq n_m$ and $P(z)$ a matrix polynomial with $\det P$ having a zero only at $z=0$. Applying Lemma \ref{L:fact_3} we can write $D(z)LP(z)=G(z)D(z)P(z)$ where $G(z)$ is a lower triangular minus matrix function with ones on the main diagonal. The result of Theorem 2.1 in Chapter 1 of \cite{CG} follows using this adaptation.


There is some divergence from the situation where $\Om$ has no poles on $\BT$. In that case, from Theorem XXIV.4.3 in \cite{GGK2} we have that $T_\Om$ is Fredholm if and only if $\det \Om (z) \not = 0$ for $z\in\BT$. When poles on $\BT$ are allowed  there could be pole-zero cancellation in the determinant of $\Om(z)$ for $z\in\BT$, for example $\Om (z) = \diag (\frac{z+1}{z-1}, \frac{z-1}{z+1})$ has $\det \Om(z) \equiv 1$. However, $T_\Om$ is not Fredholm. Thus, there are cases where $\det \Om (z) \not = 0$ for $z\in\BT$ but $T_\Om$ is not Fredholm, which does not happen in the case where $\Om$ has no poles on $\BT$.\smallskip

\paragraph{\bf Overview.}
The paper is organized as follows:\ Besides the current introduction, the paper consists of six sections. In Section \ref{S:BasicM} we prove basic results concerning the Toeplitz-like operator $T_\Om$. In the following section, Section \ref{S:PolyFact}, we derive various factorization results required for the proof of the Wiener-Hopf type factorization of Theorem \ref{T:fact}. The proofs of Theorems \ref{T:fact} and \ref{T:factUnique} will be given in Section \ref{S:FactM}, followed by an example that illustrates the construction of the Wiener-Hopf type factorization in Section \ref{S:ExampleM}. In the next section, we prove the factorization of the Toeplitz-like operator, i.e., Theorem \ref{T:Fact_Toeplitz}. Finally, in Section \ref{S:FredholmM} we consider the Fredholm properties of $T_\Om$, including a proof of Theorem \ref{T:Fredholm}, and we present some examples that exhibit the non-uniqueness in our Wiener-Hopf type factorization.

\section{Basic properties of $T_\Om$}\label{S:BasicM}

Using similar arguments as in the scalar case, cf., \cite{GtHJR1}, we determine various basic properties of the Toeplitz-like operator $T_\Om$. Some of these results can be derived by restricting to the entries of $\Om$, in which case we give minimal details of the proof. We start with an analogue of Proposition 2.1 of \cite{GtHJR1}.


\begin{proposition}\label{P:basic1}
Let $\Om\in\Rat^{m\times m}$, possibly with poles on $\BT$. Then $T_\Om$ is a well-defined, closed, densely defined linear operator on $H_m^p$.  More specifically, $\cP^m \subset \Dom (T_\Om)$. Moreover, $\Dom (T_\Om)$ is invariant under the forward shift $S_+ = T_{zI_m}$ on $H_m^p$ and we have
\begin{equation}\label{Eq:invar}
S_- T_\Om S_+ f = T_\Om f \qquad \textrm{for all } f\in\Dom (T_\Om),
\end{equation}
 where $S_- = T_{z^{-1}I_m}$ on $H_m^p$.
\end{proposition}

We first give two lemmas, without proof, which can be derived in a way analogous to the scalar case, starting with the analogue of Lemma 2.4 in \cite{GtHJR1}.

\begin{lemma}\label{L:SumDec}
Given $\Om\in\Rat^{m\times m}$, we can write $\Om(z)=\Om_1(z)+\Om_2(z)$, where $\Om_1 \in \Rat^{m\times m}$ with no poles on $\BT$ and $\Om_2\in\Rat_0^{m\times m}(\BT)$.
\end{lemma}

The above lemma allows one to reduce certain questions to the case where $\Om\in\Rat^{m\times m}(\BT)$. In that case, with arguments similar to the ones used in the proof of Lemma 2.3 in \cite{GtHJR1}, the domain can be described as in the next lemma.

\begin{lemma}
Let $\Om\in\Rat^{m\times m}(\BT)$. Write $\Om=q(z)^{-1}P(z)$ with $P\in\cP^{m\times m}$ and $q\in\cP$, $q$ having roots only on $\BT$. Then
\[
\Dom (T_\Om) = \left\{ g\in H_m^p \colon \Om g = h + q^{-1}r,\mbox{ with }h\in H_m^p,\, r\in\cP^m_{\degr(q) - 1} \right\},
\]
and $T_\om g =h$ for $g\in\Dom (T_\Om)$.
\end{lemma}

\begin{proof}[\bf Sketch of the proof of Proposition \ref{P:basic1}]
The proof mostly follows by direct generalization of the arguments in \cite{GtHJR1} and \cite{GtHJR2}, sometimes reducing to results for $T_{\om_{ij}}$, where $\Om = [\om_{ij}]_{i,j = 1}^m$.

For $\rho\in\Rat_0^m(\BT)$, using a similar argument as in Lemma 2.2 in \cite{GtHJR1} on its entries, it follows that $\rho$ is identically zero whenever $\rho\in L_m^p$. Now following the argument in Proposition 2.1 in \cite{GtHJR1} one obtains that $T_\Om$ is well-defined.

Let $\Om\in\Rat^{m\times m}$. By Lemma \ref{L:SumDec}, we can write $\Om(z)=\Om_1(z)+\Om_2(z)$, where $\Om_1 \in \Rat^{m\times m}$ with no poles on $\BT$ and $\Om_2\in\Rat_0^{m\times m}(\BT)$. Then $T_\Om=T_{\Om_1}+T_{\Om_2}$ and the domains of $T_\Om$ and $T_{\Om_2}$ coincide.  To see this, note that $f\in \Dom(T_\Om)$ if and only if $f\in \Dom(T_{\Om_2})$ and that the latter is the case if and only if $\Om_2 f= h_2+\rho$ where $h_2\in L_m^p$ and $\rho \in\Rat_0^{m}(\BT)$. Now for such a function $f$ consider $\Om f=\Om_1 f+\Om_2 f$. Since $\Om_1\in  L_\infty^{m\times m}$, also $\Om_1 f\in L_m^p$. Moreover, we have
$$
\Om f=\Om_1 f+\Om_2 f =(\Om_1 f+ h_2) +\rho = h+\rho,
$$
where $h=\Om_1 f+ h_2$. Now
$$
\BP h=\BP (\Om_1 f)+\BP h =T_{\Om_1}f+T_{\Om_2} f=
T_\Om f
$$
as desired. Hence, for various qualitative properties of $T_\Om$, including closedness, we may assume without loss of generality that $\Om\in\Rat_0^{m\times m}(\BT)$.

Assume $\Om\in\Rat_0^{m\times m}(\BT)$. Then we can write $\Om(z) = q(z)^{-1}P(z)$ where $q\in\cP_\ell$ has  zeroes  only on  $\BT$ and $P\in\cP_{\ell - 1}^{m\times m}$ for some $\ell\in\BN$. Using a similar argument as in the proof of Lemma 2.3 in \cite{GtHJR1} we can show that  $f \in \Dom(T_\Om)$ if and only if  $\Om f =h+ q^{-1}r$, where $h\in H_m^p$ and $r\in\cP^m_{\ell - 1}$. Moreover, $r$ and $h$ are unique, and in that case $T_\Om f = h$.
Now using a similar argument as in the proof of Proposition 2.1 in \cite{GtHJR1} it follows that $T_\Om$ is closed and that the domain of $T_\Om$ contains all the polynomials and so $T_\Om$ is densely defined.

Finally, to prove $\Dom(T_\om)$ is invariant under $S_+$ as well as \eqref{Eq:invar}, let $f\in\Dom(T_\Om)$. Then $\Om f=h+\rho$ for $h\in L^p_m$ and $\rho\in\Rat^m_0(\BT)$. Then $\Om S_+ f=\Om zf = zh+z\rho$. Apply the Euclidean algorithm entrywise to write $z\rho= \rho'+c$ with $\rho'\in\Rat^m_0(\BT)$, with the same poles as $\rho$, and $c\in\BC^m$. Hence $\Om S_+ f = (zh +c)+\rho'\in L^p_m+ \Rat_0^m(\BT)$. Thus $S_+ f\in\Dom(T_\om)$, and we have
\begin{align*}
S_-T_{\Om}S_+ f
&= S_- \BP(zh +c)=S_-(\BP (zh) +c)=S_- \BP z h=\BP z^{-1} \BP z h\\
&=\BP z^{-1} \BP z h + \BP z^{-1} (I_{L^p_m}-\BP) z h=\BP z^{-1} z h =\BP h=T_\Om f,
\end{align*}
where we used that $\BP z^{-1} (I_{L^p_m}-\BP)g=0$ for all $g\in L^p_m$.
\end{proof}

In order to determine the Fredholm properties of $T_\Om$, via the factorization of Theorem \ref{T:fact}, we can reduce to the case of a diagonal matrix function in $\Rat^{m\times m}(\BT)$, with zeroes all on $\BT$. Therefore we will not attempt here to give an explicit description of the kernel, range and domain for the case $\Om\in\Rat^{m\times m}(\BT)$ in the form of an analogue of Theorem 2.2 in \cite{GtHJR2}. For the diagonal matrix case, results are easily obtained by reduction to the scalar case. Here, and in the sequel, we shall identify the direct sum $H^p \oplus \cdots \oplus H^p$ of $m$ copies of $H^p$ with $H_m^p$, and likewise for $L^p$.

\begin{proposition}\label{P:diagonal}
Suppose that $\Om\in\Rat^{m\times m}$ is of the form
$$
\Om(z) = \diag(\om_1(z),\ldots,\om_m(z)),\quad \mbox{with $\om_j\in\Rat, j = 1, 2, \ldots m$}
$$
Then
\begin{equation}\label{DiagDomain}
\Dom(T_\Om) = \Dom(T_{\om_1}) \oplus \Dom(T_{\om_2}) \oplus \cdots \oplus \Dom (T_{\om_m}),
\end{equation}
and for $f=f_1\oplus \cdots\oplus f_m\in\Dom(T_\Om)$ we have $T_\Om f= T_{\om_1} f_1\oplus \cdots\oplus T_{\om_m} f_m$.
Furthermore, we have
\begin{align*}
\Ran(T_\Om) &= \Ran(T_{\om_1}) \oplus \Ran(T_{\om_2}) \oplus \cdots \oplus \Ran (T_{\om_m});\\
\kernel(T_\Om) &= \kernel(T_{\om_1}) \oplus \kernel(T_{\om_2}) \oplus \cdots \oplus \kernel (T_{\om_m}).
\end{align*}
\end{proposition}

\begin{proof}[\bf Proof]
For $j=1,\ldots,m$, suppose that $f_j\in\Dom (T_{\om_j})$. Then $\om_j f_j = h_j + \rho_j$ with $h_j\in L^p$ and $\rho_j\in\Rat_0(\BT)$ so that
$$
\Om f = h + \rho
$$
where
\begin{equation}\label{fhrho}
f = \begin{pmatrix} f_1\\ \vdots \\ f_m \end{pmatrix}\in H^p_m, \quad h = \begin{pmatrix} h_1\\ \vdots \\ h_m \end{pmatrix}\in L^p_m, \quad \rho = \begin{pmatrix} \rho_1\\ \vdots \\ \rho_m \end{pmatrix}\in\Rat^m_0(\BT).
\end{equation}
Thus $f\in \Dom (T_\Om)$ and we have $T_\Om f= T_{\om_1} f_1\oplus \cdots\oplus T_{\om_m} f_m$. It follows that
\[
\Dom(T_{\om_1}) \oplus \Dom(T_{\om_2}) \oplus \cdots \oplus \Dom (T_{\om_m}) \subset \Dom (T_\Om).
\]

To show the converse inclusion, suppose that $f\in \Dom (T_\Om)$. Then there are $h\in L_m^p$ and $\rho\in\Rat_0^m(\BT)$ with $\Om f = h + \rho$. Decomposing $f$, $h$ and $\rho$ as in \eqref{fhrho}, it follows that $\om_j f_j = h_j + \rho_j$ for each $j$, showing that $f_j\in \Dom (T_{\om_j})$. Hence \eqref{DiagDomain} holds and it follows that the action of $T_\Om$ relates to the action of the operators $T_{\om_j}$, $j=1,\ldots,m$, as claimed.

The formulas for the range and kernel of $T_\Om$ are now straightforward.
\end{proof}

Getting an explicit formulation of the domain, range and kernel of $T_\Om$ beyond the diagonal case is much more complicated than in the scalar case, even when $\Om\in\Rat^{m\times m}(\BT)$. We indicate the difficulty in the following lemma.

\begin{lemma}\label{L:domains}
Let $\Om\in\Rat^{m\times m}$ and write $\Om = \Om_1 + \Om_2$ where $\Om_1 \in \Rat^{m\times m}$ with no poles on $\BT$ and $\Om_2\in\Rat_0^{m\times m}(\BT)$, so that $\Dom (T_\Om) = \Dom (T_{\Om_2})$. Let $\Om_2 = q^{-1}P$ with $q\in\cP$ with zeroes only on $\BT$ and $P\in\cP^{m\times m}$ so that no root of $q$ is also a root of each entry of $P$. Suppose $\Om_2=\sbm{\frac{s_{ij}}{q_{ij}}}_{i,j=1}^m$ with $s_{ij}, p_{ij}\in\cP$ co-prime for all $i$ and $j$  and let $q_j$ be the least common multiple of $q_{1j},\ldots,q_{mj}$. Then
\begin{equation}\label{Eq:domains}
qH_m^p + \cP_{\degr q - 1}^m \subset \oplus_{j=1}^m \left ( q_jH^p + \cP_{\degr q_j - 1} \right ) \subset \Dom (T_\om)
\end{equation}
and both inclusions can be strict.
\end{lemma}

\begin{proof}[\bf Proof]
The first inclusion follows since the roots of each $q_j$ will be included in the zeroes of $q$, multiplicities taken into account, which gives $q H^p\subset q_j H^p$. Since $\cP\subset qH^p + \cP_{\degr q - 1}$ and $\cP\subset q_j H^p + \cP_{\degr q_j - 1}$, we obtain
\[
qH^p + \cP_{\degr q - 1} \subset q_j H^p + \cP_{\degr q_j - 1}, \mbox{ for }j=1,\ldots,m.
\]

For the second inclusion, let $f=f_1\oplus \cdots \oplus f_m\in H^p_m$ and suppose $f_j = q_j h_j + r_j\in q_jH^p + \cP_{\degr q_j - 1}$. Then $q_j = u_{ij}q_{ij}$ for some polynomial $u_{ij}$. Write $s_{ij}r_j=q_{ij}r_{ij} + \widetilde{r}_j$ for some $\widetilde{r}_j\in\cP_{\degr q_{ij}-1}$ and $r_{ij}\in\cP$.
Then
\[
s_{ij}f_j = s_{ij}q_jh_j + s_{ij}r_j =q_{ij}(s_{ij}u_{ij}h_j + r_{ij})+\widetilde{r}_j.
\]
Since the $i$-th entry in $\Om_2 f$ is given by $\sum_{j=1}^m \frac{s_{ij}}{q_{ij}}f_j$, we have
\begin{align*}
(\Om_2 f)_i &= \sum_{j=1}^n \frac{s_{ij}}{q_{ij}}f_j
= \sum_{j=1}^m \left (s_{ij} u_{ij}h_j + r_{ij}\right ) + \sum_{j=1}^m \frac{\widetilde{r}_j}{q_{ij}}.
\end{align*}
Note that $\wtil{r}_j/q_{ij}\in \Rat_0(\BT)$ for each $j$. Then also $\sum_{j=1}^m \wtil{r}_j/q_{ij}\in \Rat_0(\BT)$ for each $j$. This proves that $\Om f \in H^p_m + \Rat_0^m(\BT)$, and thus $f\in\Dom(T_\Om)$.

It is not difficult to construct examples where the first inclusion is strict, use for instance Lemma 3.5 in \cite{GtHJR1}. To see that the second inclusion can be strict, consider Example \ref{E:strict} below.
\end{proof}

\begin{example}\label{E:strict}
Consider $\Om\in\Rat^{2 \times 2}(\BT)$ given by
\[
\Om (z) = \begin{pmatrix}
\frac{z}{z-1} & \frac{1}{z-1}\\ \frac{1}{z+1} &\frac{z+2}{z+1}
          \end{pmatrix}.
\]
Take $f=f_1 \oplus -f_1$ with $f_1\in H^p$ arbitrarily. Then $\Om(z) f(z)= f_1\oplus -f_1\in H_2^p$, hence $f\in \Dom(T_\Om)$. In this case, the greatest common divisor of the columns of $\Om$ is $q_1(z)=z^2-1$. By Lemma 3.5 in \cite{GtHJR1}, there exist $f_1\in H^p$ which are not in $(z-1)H^p + \BC$ (or not in $(z+1)H^p + \BC$) and so $f_1\not\in (z^2-1)H^p + \cP_1$. Selecting $f_1$ in such a way, it follows that $f$ is not in $((z^2-1)H^p + \cP_1) \oplus ((z^2-1)H^p + \cP_1)$, proving the second inclusion in \eqref{Eq:domains} can be strict.
\end{example}

\section{Matrix polynomial factorization}\label{S:PolyFact}

In this section we prove a few factorization results for matrix polynomials that will be of use in the sequel. The Smith decomposition for matrix polynomials plays a prominent role in our construction, hence for the readers convenience we will list a variation on it here; see Gantmacher \cite{G59} or Gohberg-Lancaster-Rodman \cite{GLR82} for a proof. For simplicity, since we only encounter this case, we only consider the case of square matrix polynomials whose determinant is not uniformly zero.

\begin{theorem}\label{T:Smith}
Let $R\in\cP^{m \times m}$ with $\det R(z)\not\equiv 0$. Then we can write
\begin{equation}\label{SD}
R(z) = E(z) D(z) F(z)
\end{equation}
where $E,F$ are matrix polynomials with nonzero constant determinants and $D$ is a diagonal matrix polynomial that factors as
\[
D(z) = D_-(z) D_\circ(z) D_+(z)
\]
with $D_-, D_\circ$ and $D_+$ also diagonal matrix polynomials with zeroes in $\BD$, on $\BT$ and outside $\ov{\BD}$, respectively. Moreover, the diagonal matrix polynomial $D=\diag(d_1,\ldots,d_m)$ can be chosen in such a way that all diagonal entries are monic and $d_{j+1}$ is a factor of $d_j$ for $j=1,\ldots,m-1$, and with these additional conditions the diagonal entries $d_1,\ldots, d_m$ are uniquely determined by $R$ and are given by
\[
d_j(z)=\frac{D_j(z)}{D_{j-1}(z)}, \quad j=1,\ldots,m,
\]
where $D_0(z)\equiv 1$ and for $r>0$, $D_r$ is the greatest common devisor of all minors of $R$ of order $r$. Furthermore, the diagonal entries of $D_-$, $D_\circ$ and $D_+$ can be taken monic and ordered with respect to factorization, as was done with $D$, and then $D_-$, $D_\circ$ and $D_+$ are also uniquely determined by $R$.
\end{theorem}

Note that the assumption $\det R(z)\not\equiv 0$ implies that the minors of a given order $r$ cannot all be 0, so that $D_r$ is well defined and not equal to the zero polynomial.

In part of the proofs we require a slightly more refined version of the Smith form, which we shall call the regional Smith form, and which subsumes both the classical Smith form of Theorem \ref{T:Smith} and the local Smith form (see Theorem S1.10 in \cite{GLR82}); indeed the global and local versions appear in case $\La=\BC$ and $\La=\{z_0\}$ for some $z_0\in\BC$, respectively.

\begin{theorem}\label{T:SmithReg}
Let $R\in\cP^{m \times m}$ with $\det R(z)\not\equiv 0$ and let $\Lambda\subset \BC$. Then we can write
\begin{equation}\label{SDreg}
R(z) = E_\La(z) D_\La(z) F_\La(z)
\end{equation}
where $E_\La,F_\La$ are matrix polynomials which are invertible for all $z\in\La$ and $D_\La=\diag(p_1,\ldots,p_m)$ a diagonal matrix polynomial such that $p_j$, $j=1,\ldots,m$, has roots only in $\La$. Furthermore, the polynomials $p_j$, $j=1,\ldots,m$ can be chosen to be monic and in such a way that $p_{j+1}$ is a factor of $p_j$, for $j=1,\ldots,m-1$, and with this additional conditions the diagonal entries $p_1,\ldots, p_m$ are uniquely determined by $R$ and $\La$.
\end{theorem}

We refer to the diagonal matrix polynomial $D$ in Theorem \ref{T:SmithReg}, made unique through the assumptions that the diagonal entries be monic and the division property for subsequent diagonal entries, as the {\em Smith form of $R$ with respect to $\La$}, and any factorization of the type \eqref{SDreg} with properties as listed in the theorem as a {\em Smith decomposition of $R$ with respect to $\La$}. In the classical case, with $\La=\BC$, we simply speak of the Smith form and Smith decomposition of $R$, without referring to the region.

\begin{proof}[\bf Proof of Theorem \ref{T:SmithReg}]
The existence of a Smith decomposition of $R$ with respect to $\La$ follows directly from the classical Smith decomposition of Theorem \ref{T:Smith}. Write $R(z)$ as in \eqref{SD} with $D(z)=\diag_{j=1}^m(d_j(z))$. Factor each $d_j$ as $d_j(z)=p_j(z)q_j(z)$ with $p_j,q_j\in\cP$, $p_j$ monic and having roots only in $\La$ and $q_j$ having roots only outside $\La$. Now set $D_\La(z)=\diag_{j=1}^m(p_j(z))$, $E_\La=E$ and $F_\La=\diag_{j=1}^m(q_j(z)) F(z)$, but the roots of $d_j$ outside $\La$ can be distributed over $E$ and $F$ in any other way. It is clear that $E_\La$, $D_\La$ and $F_\La$ have the required properties, and that the division property of the diagonal entries of $D$ caries over to $D_\La$ in case the polynomials of $D$ are ordered as in Theorem \ref{T:Smith}.

It remains to prove the uniqueness claim. As in the case of the proof of the local Smith form in \cite[Theorem S1.10]{GLR82}, this relies on the Cauchy-Binet formula, cf.,  Subsection 0.8.7 in \cite{HJ85}. Assume that in addition to the factorization \eqref{SDreg} with properties as listed in the theorem, including the division property of the diagonal entries $p_1,\ldots,p_m$ of $D$, $R$ also admits a second Smith decomposition $R(z)=\wtil{E}_\La(z)\wtil{D}_\La(z)\wtil{F}_\La(z)$ with respect to $\La$, with $\wtil{D}_\La(z)=\diag(\wtil{p}_1(z),\ldots,\wtil{p}_m(z))$ and $\wtil{p}_{j+1}$ a factor of $\wtil{p}_{j}$ for $j=1,\ldots,m-1$. Define $\Phi(z):=E_\La(z)^{-1}\wtil{E}_\La(z)$ and $\Upsilon(z):=\wtil{F}_\La(z)F_\La(z)^{-1}$. Then $\Phi,\Up\in\Rat^{m\times m}$ have no poles or zeroes in $\La$. Hence also all the minors of $\Phi,\Up,\Phi^{-1},\Up^{-1}$ have no poles in $\La$. We have
\begin{equation}\label{DiagFactDiag}
D_\La(z)=\Phi(z) \wtil{D}_\La(z)\Up(z) \ands
\wtil{D}_\La(z)=\Phi(z)^{-1}D_\La(z)\Up(z)^{-1}.
\end{equation}
Write $\un{m}=\{1,\ldots m\}$ and for $L,S\subset \un{m}$ with $\#(L)=\#(S)$ and $M$ an $m\times m$ matrix, write $|M|_{L,S}$ for the minor obtained by selecting only the rows indexed by the entries of $L$ and only the columns indexed by the entries of $S$. Fix $k\in\un{m}$ and set $L=\{m-k,\ldots,m\}$. By the Cauchy-Binet formula
\begin{align*}
p_{m-k}(z)\cdots p_m(z)
& = |D_\La(z)|_{L,L}
=|\Phi(z) \wtil{D}_\La(z)\Up(z)|_{L,L}\\
&=\sum_{S\subset\un{m},\, \#(S)=k} |\Phi(z)|_{L,S} |\wtil{D}_\La(z)\Up(z)|_{S,L} \\
&=\sum_{S\subset\un{m},\, \#(S)=k} |\Phi(z)|_{L,S} |\wtil{D}_\La(z)|_{S,S} |\Up(z)|_{S,L}\\
&=\sum_{S\subset\un{m},\, \#(S)=k} |\Phi(z)|_{L,S} |\Up(z)|_{S,L}\prod_{j\in S}\wtil{p}_j(s).
\end{align*}
Since $|\Phi(z)|_{L,S}$ and $|\Up(z)|_{S,L}$ do not have poles in $\La$, $\prod_{j\in S}\wtil{p}_j(s)$ is a factor of the numerator of $|\Phi(z)|_{L,S} |\Up(z)|_{S,L}\prod_{j\in S}\wtil{p}_j(s)$ for all $S\subset\un{m}$ with $\#(S)=k$. Also, by the factorization order of the diagonal entries in $\wtil{D}_\La$ we know that $\wtil{p}_{m-k}(z)\cdots \wtil{p}_m(z)$ is a factor of $\prod_{j\in S}\wtil{p}_j(s)$ for all $S\subset\un{m}$ with $\#(S)=k$. Consequently, by the above identity we know that $\wtil{p}_{m-k}(z)\cdots \wtil{p}_m(z)$ is a factor of $p_{m-k}(z)\cdots p_m(z)$. Applying the same argument to the second identity in \eqref{DiagFactDiag}, one obtains that also $p_{m-k}(z)\cdots p_m(z)$ is a factor of $\wtil{p}_{m-k}(z)\cdots \wtil{p}_m(z)$, hence they are equal because both are monic polynomials. Since this identity holds for all $k\in\un{m}$, it follows that $p_j=\wtil{p}_j$ for $j=1,\ldots,m$.
\end{proof}

\begin{corollary}\label{C:SmithMult}
Let $R\in\cP^{m \times m}$ with $\det R(z)\not\equiv 0$ and let $\Lambda\subset \BC$. For all matrix polynomials $M,N\in\cP^{m\times m}$ which are invertible for all $z\in\La$, the matrix polynomials $R$ and $MRN$ have the same Smith form with respect to~$\La$.
\end{corollary}

\begin{proof}[\bf Proof]
Let $R(z)=E(z)D(z)F(z)$ be the factorization of Theorem \ref{T:SmithReg} with $D$ the Smith form of $R$ with respect to $\La$. Then
\[
M(z)R(z)N(z)=(M(z)E(z))D(z)(F(z)N(z))
\]
is a factorization of the type in Theorem \ref{T:SmithReg}, with $M(z)E(z)$ and $F(z)N(z)$ both being invertible for all $z\in\La$ and the diagonal polynomials of $D$ still have the required properties to guarantee uniqueness. In particular $D$ is also the Smith form of $MRN$ with respect to $\La$.
\end{proof}

The next result is used to repair an oversight in the construction in \cite{CG}.

\begin{lemma}\label{L:fact_3}
Let $F\in\cP^{m\times m}$ with $\det F(z) = z^np(z)$ for a $p\in\cP$ with $p(0)\neq 0$. Then
\begin{equation}\label{F=QR}
F(z) = Q(z) R(z)
\end{equation}
where $Q,R\in\cP^{m\times m}$ with $\det R(z)=p(z)$ and $Q$ a lower triangular matrix polynomial with $\det Q(z)=z^n$. In particular, the diagonal entries of $Q$ are of the form $z^{n_1},\ldots,z^{n_m}$, with $z^{n_j}$ on the $j$-th diagonal entry, $\sum_{j=1}^m n_j=n$, and, moreover, the indices $n_1,\ldots,n_m$ are uniquely determined by $F$. Furthermore, a factorization \eqref{F=QR} exists with the polynomial entries left of the diagonal of $Q$ having a degree lower than the degree of the diagonal entry in the same row, and with this additional condition $Q$ and $R$ are uniquely determined.
\end{lemma}

\begin{proof}[\bf Proof]
The proof follows by a recursive procedure, in four parts.\smallskip

\paragraph{\bf Part 1:\ First step}
Write
$$F(z) =
\diag(z^{n_1},\ldots, z^{n_m})\, R_1(z)
$$
where $R_1\in\cP^{m\times m}$ and  ${n_i}$ is the highest power of $z$ dividing all the entries in row $i$. Then $\det R_1(z) = z^{n'}p(z)$ with $n' = n - \sum_{i=1}^m n_i$ and $Q_1(z):=\diag(z^{n_1},\ldots, z^{n_m})$ is a lower triangular matrix polynomial with $\det Q_1(z)=z^{n-n'}$ with entries left of the diagonal equal to 0, thus of degree 0, hence they have a lower degree that the diagonal entry in the same row. If $n'=0$ we take $Q=Q_1$ and $R=R_1$ and are done.\smallskip

\paragraph{\bf Part 2:\ Second step}
In case $n'\neq 0$, write
$$
R_1(z) = \begin{pmatrix}
          r_1(z) \\ \vdots \\ r_m(z)
         \end{pmatrix},\mbox{ with $r_j\in\cP^{1 \times m},\, j=1,\ldots, m$}.
$$
Then $r_1(0),  \ldots, r_m(0)$ are linearly dependent. Let $k$ be the smallest integer such that $r_1(0), \ldots, r_k(0)$ are linearly dependent. Then there are numbers $\alpha_1, \alpha_2, \ldots , \alpha_{k-1} $ such that
$$
\alpha_1 r_1(0) +  \cdots + \alpha_{k-1} r_{k-1}(0) + r_k(0) = 0.
$$
Put
$$
L =
\begin{pmatrix}
1 & 0 & \cdots & \cdots & \cdots & \cdots & \cdots & 0\\
0 & 1 & \ddots & & & & & \vdots  \\
\vdots&\ddots& \ddots & \ddots & & & & \vdots \\
0 & \cdots & 0 & 1 & 0 & \cdots & \cdots & 0\\
\alpha_1 & \cdots & \cdots & \alpha_{k-1} & 1 & 0 & \cdots & 0 \\
0 & \cdots & \cdots & \cdots & 0 & 1 &\ddots &\vdots   \\
\vdots &  &  & &&\ddots &  \ddots & 0\\
0 & \cdots & \cdots & \cdots & \cdots & \cdots & 0 & 1
\end{pmatrix}
$$
with the $\al_j$'s appearing  in the $k$-th row. Then
$$
F(z) = \diag(z^{n_1},\ldots, z^{n_m})\, L^{-1} L R_1(z).
$$
Since $\sum_{i=1}^k \alpha_i r_i(z)$ has a root at zero, where we set $\al_k=1$, we have
$$L R_1(z) =
\begin{pmatrix} r_1(z) \\ \vdots \\ r_{k-1}(z) \\ \sum_{i=1}^k \alpha_i r_i(z) \\ r_{k+1}(z) \\ \vdots \\ r_m(z) \end{pmatrix} =
\diag(1,\ldots,1,z^\ell,1,\ldots,1)\,
R_2(z)
$$
for some $\ell \geq 1$ and $R_2\in\cP^{m\times m}$, where on the right-hand side $z^\ell$ appears in the $k$-th diagonal entry. Note that
\begin{equation}\label{Einv}
L^{-1} =
\begin{pmatrix}
1 & 0 & \cdots & \cdots & \cdots & \cdots & \cdots & 0\\
0 & 1 & \ddots & & & & & \vdots  \\
\vdots&\ddots& \ddots & \ddots & & & & \vdots \\
0 & \cdots & 0 & 1 & 0 & \cdots & \cdots & 0\\
-\alpha_1 & \cdots & \cdots & -\alpha_{k-1} & 1 & 0 & \cdots & 0 \\
0 & \cdots & \cdots & \cdots & 0 & 1 &\ddots &\vdots   \\
\vdots &  &  & &&\ddots &  \ddots & 0\\
0 & \cdots & \cdots & \cdots & \cdots & \cdots & 0 & 1
\end{pmatrix}
\end{equation}
again with the $\al_j$'s appearing in the $k$-th row. By direct computation we find
\[
\diag(z^{n_1},\ldots,z^{n_m})\, L^{-1} = G_2(z)\, \diag(z^{n_1},\ldots,z^{n_m})
\]
where
\[
G_2(z) :=
\begin{pmatrix}
1 & 0 & \cdots & \cdots & \cdots & \cdots & \cdots & 0\\
0 & 1 & \ddots & & & & & \vdots  \\
\vdots&\ddots& \ddots & \ddots & & & & \vdots \\
0 & \cdots & 0 & 1 & 0 & \cdots & \cdots & 0\\
-\alpha_1 z^{n_k-n_1} & \cdots & \cdots & -\alpha_{k-1} z^{n_k - n_{k-1}} & 1 & 0 & \cdots & 0 \\
0 & \cdots & \cdots & \cdots & 0 & 1 &\ddots &\vdots   \\
\vdots &  &  & &&\ddots &  \ddots & 0\\
0 & \cdots & \cdots & \cdots & \cdots & \cdots & 0 & 1
\end{pmatrix},
\]
with the $-\alpha_j z^{n_k-n_j}$ entries in the $k$-th row. We now have
$$
F(z) = G_2(z)\, \diag(z^{n_1},\ldots,z^{n_{k - 1}},z^{n_k + \ell},z^{n_{k + 1}},\ldots,z^{n_m})\, R_2(z),
$$
and $Q_2(z):=G_2(z)\, \diag( z^{\wtil{n}_1},\ldots, z^{\wtil{n}_m})\in\cP^{m\times m}$, where we set $\wtil{n}_j = n_j$ if $j\not= k$ and $\wtil{n}_k = n_k + \ell$. Note that $\det R_2(z) = z^{n' - \ell}p(z)$, while the entries in the $k$-th row of $Q_2$ left of the diagonal have degree $n_k$ and the diagonal entry in the $k$-th row has degree $n_k+\ell> n_k$ and all other off-diagonal entries are 0. In case $n'-\ell=0$ (equivalently, $\sum_{j=1}^m \wtil{n}_j=n$), take $Q=Q_2$ and $R=R_2$ and we are done.\smallskip

\paragraph{\bf Part 3:\ Recursion}
In case $n'-\ell\neq 0$, repeat the above construction starting with $R_2$ instead of $R_1$. In each step the power of $z$ in the determinant of $R_j$ decreases, hence the process stops after at most $n$ steps. It remains to see that the entries left of the diagonal have the required restriction on the degree. To see that this is the case, we claim that in the $j$-th step, going from factorization $F(z)=Q_{j-1}(z)R_{j-1}(z)$ to $F(z)=Q_{j}(z)R_{j}(z)$, all entries of $Q_j$ left of the diagonal have a degree lower than the degree of the diagonal entry in the same row and, in addition, if $k$ is the first integer so that rows 1 to $k$ in $R_{j-1}(0)$ are linearly dependent, then in $Q_j$ all off-diagonal entries in rows $k+1$ to $m$ are 0. These properties certainly hold in the first two steps. Now assume this is satisfied in the step leading to the factorization $F(z)=Q_{j-1}(z)R_{j-1}(z)$. Assume $k$ is the first integer so that rows 1 to $k$ in $R_{j-1}(0)$ are linearly dependent. From the procedure it follows that in the previous step, the first occurrence of linear dependence in the rows of $R_{j-2}(0)$ must also have been in rows 1 to $k$. Hence, by assumption, in $R_{j-1}$ the off-diagonal entries in rows $k+1$ to $m$ are all 0. Then $Q_j$ is obtained by multiplying $Q_{j-1}$ with a matrix $L^{-1}$ of the form \eqref{Einv} on the right and then with $\diag(1,\ldots,1,z^\ell, 1,\ldots,1)$ also on the right, where $z^\ell$ appears in the $k$-th entry. One easily checks that rows 1 to $k-1$ of $Q_{j-1}$ and $Q_{j}$ coincide, due to the lower triangular structure, and that rows $k+1$ to $m$ of $Q_{j-1}$ and $Q_{j}$ coincide, due to the zeros in the off-diagonal entries in the rows $k+1$ to $m$. In particular, it follows from the above arguments for all but the $k$-th row that the entries left of the diagonal have a degree less than the diagonal entry in the same row, while all off-diagonal entries in rows $k+1$ to $m$ remain 0. Assume the entries of $Q_{j-1}$ left of the diagonal are given by $q_{i,j}$, $i>j$ and that the $i$-th diagonal entry is $z^{n_i}$, so that $\degr q_{i,j}< n_i$. Then, on the $k$-th row, left of the diagonal we obtain entries of the form $q_{k,j}(z)- \al_j z^{n_j}$ with degree at most $n_j$, while the diagonal entry becomes $z^{n_j+\ell}$, which proves our claim.\smallskip

\paragraph{\bf Part 4: Uniqueness}
We first show that the diagonal entries of $Q$ are unique, without assuming additional degree constraints on the entries left of the diagonal. Suppose that there is another factorization of $F$ of the same type, i.e.,
\[
Q(z)R(z)=F(z) = Q'(z) R'(z),
\]
with $Q',R'\in\cP^{m\times m}$, $\det R'(z)=p(z)$ and $Q'$ lower triangular so that $\det Q'(z)=z^n$. Assume the $j$-th diagonal entries of $Q$ and $Q'$ are $z^{n_j}$ and $z^{s_j}$, respectively, for $j=1,\ldots,m$. Note that $Q^{-1}(z)$ is in $\Rat^{m\times m}$, lower triangular with $z^{-n_j}$ on the $j$-th diagonal entry, so that
\begin{equation}\label{lowertrian}
R(z)(R')^{-1}(z)  = Q^{-1}(z)Q'(z)
 = \begin{pmatrix}
z^{s_1 - n_1} & 0 & \cdots & 0\\
* & z^{s_2 - n_2} & \ddots &\vdots\\
\vdots &\ddots &\ddots  & 0\\
* & \cdots  &  *  & z^{s_m - n_m}
\end{pmatrix}.
\end{equation}
We have $\sum_{j=1}^m s_j=\sum_{j=1}^m n_j=n$, so we are done if we can prove that $s_j \geq n_j$ for all $j$. To see that this is the case, multiply \eqref{lowertrian} with $p$. Since $\det R'(z)=p(z)$, we have $p(z)(R')^{-1}(z)\in\cP^{m\times m}$. Hence, the left hand side in \eqref{lowertrian}, multiplied with $p$, is in $\cP^{m \times m}$. Consequently, also the right hand side in \eqref{lowertrian}, multiplied with $p$, is in $\cP^{m \times m}$. Note that the diagonal entries are of the form $p(z)z^{s_j-n_j}$ and must be in $\cP$. Since $p(0)\neq 0$, this implies $s_j\geq n_j$ for all $j$, as claimed. Thus $n_j=s_j$ for all $j$. It follows that the diagonal entries of $Q$ are uniquely determined.

Write $q_{i,j}$ and $q'_{i,j}$ for the $(i,j)$-th entries of $Q$ and $Q'$, respectively. Now assume $\degr q_{i,j}<n_j$ and $\degr q'_{i,j}<s_j=n_j$. Set $\wtil{R}(z):=R(z)(R')^{-1}(z)\in\Rat^{m\times m}$. We observed above that $\wtil{R}$ is lower triangular with diagonal entries equal to 1, since $s_j=n_j$ for all $j$, and $p(z)\wtil{R}\in\cP^{m\times m}$, so that the entries of $\wtil{R}$ left of the diagonal have the form $\wtil{r}_{i,j}(z)/p(z)$ for a $\wtil{r}_{i,j}\in\cP$ for the $(i,j)$-th entry, with $i>j$. Also, we have $Q(z)\wtil{R}(z)=Q'(z)$. We claim that $\wtil{R}(z)=I_m$ for all $z$, which proves our claim. To see this, we need to show $\wtil{r}_{i,j}=0$ for all $i>j$. Fix $j\in\{1,\ldots,m\}$. Then for $i> j$ the identity $Q(z)\wtil{R}(z)=Q'(z)$ yields
\[
q'_{i,j}(z)=q_{i,j}(z)+\sum_{k=j+1}^{i-1} \frac{q_{i,k}(z)\wtil{r}_{k,j}(z)}{p(z)} + \frac{\wtil{r}_{i,j}(z)z^{n_i}}{p(z)}.
\]
Recall that $p(0)\neq 0$, so $p(z)$ and $z^{n_i}$ have no common factor. First take $i=j+1$. In that case we find that $q'_{j+1,j}(z)-q_{j+1,j}(z)=\wtil{r}_{j+1,j}(z) z^{n_{j+1}}/p(z)$. Assume $\wtil{r}_{j+1,j}\neq 0$.  Since $p$ and $z^{n_{j+1}}$ have no common factor, the right hand side is a polynomial of degree at least $n_{j+1}$. However, by assumption the degree of the polynomial on the left hand side is less than $n_{j+1}$, leading to a contradiction. Hence $\wtil{r}_{j+1,j}=0$. Now consider $i=j+2$. Since $\wtil{r}_{j+1,j}=0$, the above identity for $q'_{j+2,j}$ reduces to $q'_{j+2,j}(z)-q_{j+2,j}(z)=\wtil{r}_{j+2,j}(z) z^{n_{j+2}}/p(z)$, and a similar argument shows that $\wtil{r}_{j+2,j}=0$. Proceeding this way, one obtains
$\wtil{r}_{i,j}=0$ for all $i>j$, which completes the proof.
\end{proof}

\begin{lemma}\label{L:invariantfact}
Let $P\in\cP^{m\times m}$ and $N=\degr P$, so that $\widetilde{P}(z) := z^N P(\frac{1}{z})$ is in $\cP^{m\times m}$.
Furthermore, let $P(z) = E(z)D(z)F(z)$ and $\widetilde{P}(z) = \wtil{E}(z)\wtil{D}(z)\wtil{F}(z)$ be the Smith decompositions of $P$ and $\wtil{P}$, so that the diagonal elements $p_1,\ldots,p_m$ of $P$ and $\wtil{p}_1,\ldots,\wtil{p}_m$ of $\wtil{P}$ are ordered as in Theorem \ref{T:Smith}. Let $\al\neq 0$. Then $d_j$ has a root at $\al$ of order $k$ if and only if $\wtil{d}_j$ has a root of order $k$ at $\al^{-1}$.
\end{lemma}

\begin{proof}[\bf Proof]
Let $P$, $\wtil{P}$ and their Smith decompositions be as stated in the lemma. Recall from Theorem \ref{T:Smith} that for $j=1,\ldots,m$ we have
\[
d_j(z) = \frac{D_j(z)}{D_{j-1}(z)}\ands
\wtil{d}_j(z) = \frac{\wtil{D}_j(z)}{\wtil{D}_{j-1}(z)}
\]
with $D_r$ and $\wtil{D}_r$ the g.c.d.\ of all minors or order $r$ of $P$ and $\wtil{P}$, respectively. The relation $\widetilde{P}(z) = z^N P(\frac{1}{z})$ translates in terms of the g.c.d.\ of the minors as $\wtil{D}_j(z)=z^{N j} D_j(\frac{1}{z})$, so that
\[
\wtil{d}_j(z) = \frac{\wtil{D}_j(z)}{\wtil{D}_{j-1}(z)} = \frac{D_j(\frac{1}{z})z^{Nj}}{D_{j-1}(\frac{1}{z})z^{N(j-1)}} = z^N d_j (\mbox{$\frac{1}{z}$}).
\]
From this it directly follows for $\al\neq 0$ that $d_j$ has a root of order $k$ at $\al$ if and only if $\wtil{d}_j$ has a root of order $k$ at $\alpha^{-1}$.
\end{proof}

\section{The Wiener-Hopf type factorization}\label{S:FactM}

In this section we prove the Wiener-Hopf type factorization presented in Theorem \ref{T:fact} as well as the uniqueness claims of Theorem \ref{T:factUnique}. In fact, we give a construction for how such a factorization can be obtained. The construction relies strongly the ideas of the construction in Chapter 2 of \cite{CG} for contours more general than the unit circle, but without the possibility of having poles on the contour.

Using the polynomial factorization results of Section \ref{S:PolyFact}, we are now ready to prove Theorem \ref{T:fact}.

\begin{proof}[\textbf{Proof of Theorem \ref{T:fact}}]
Let $\Om\in\Rat^{m\times m}$ with $\det \Om(z)\not\equiv 0$. The proof is an adaptation of the proof of Theorem 2.1 in \cite{CG} and will be divided into four steps.\smallskip

\paragraph{Step 1} Firstly, let $q$ be the least common multiple of the denominators of the matrix entries of $\Om$, so that $q(z)\Om(z)\in\cP^{m \times m}$. As in Lemma 5.1 in \cite{GtHJR1}, write $q^{-1}(z) =  z^\kappa \om_-(z) \om_\circ(z)  \om_+(z)$ where $\om_-(z)$ and $\om_-(z)^{-1}$ are minus functions, $\om_+(z)$ and $\om_+(z)^{-1}$ are plus functions and $\om_\circ(z)$ has zeroes and poles only on $\BT$. Note that $\kappa$ is uniquely determined by $q$, while the factors $\om_-, \om_\circ,\om_+$ are uniquely determined up to a nonzero constant. In fact, if $q(z)=q_-(z)q_\circ(z)q_+(z)$ with $q_-,q_\circ,q_+\in\cP$ the factors of $q$ with roots only inside $\BD$, on $\BT$ and outside $\overline{\BD}$, respectively, then $\kappa=-\degr q_-$ and, up to a nonzero constant, $\om_-(z)=z^{-\kappa}/q_-(z)$, $\om_\circ(z)=1/q_\circ(z)$ and $\om_+(z)=1/q_+(z)$.\smallskip

\paragraph{Step 2} Define $P_1(z):=q(z)\Om(z)\in\cP^{m\times m}$ and factor $P_1$ as in the (extended) Smith decomposition of Theorem \ref{T:Smith}:
$$P_1(z) = E_1(z) D_1^-(z) D_1^\circ(z) D_1^+(z) F_1(z),$$
where $E_1$ and $F_1$ are matrix polynomials with nonzero constant determinants, and $D_1^-$, $D_1^\circ$ and $D_1^+$ are diagonal matrix polynomials, with roots only inside $\BD$, on $\BT$ and outside $\overline{\BD}$, respectively, with in all three the diagonal entries monic and ordered as in Theorem \ref{T:Smith}. Note that $D_1^+(z) F_1(z)$ is a matrix polynomial with all its roots outside $\overline{\BD}$, hence it is a plus function whose inverse is also a plus function.\smallskip

\paragraph{Step 3} Note that $E_1(z) D_1^-(z) D_1^\circ(z)$ is a polynomial that has all its zeroes in $\overline{\BD}$. Let $N=\degr E_1(z) D_1^-(z) D_1^\circ(z)$ and define
\[
P_2(z) :=z^N E_1(\mbox{$\frac{1}{z}$}) D_1^-(\mbox{$\frac{1}{z}$}) D_1^\circ(\mbox{$\frac{1}{z}$})\in\cP^{m\times m},
\]
so that
\[
P_1(z) = z^{N} P_2(\mbox{$\frac{1}{z}$}) D_1^+(z) F_1(z).
\]
Then the zeroes of $P_2$ can only be on $\BT$, outside $\overline{\BD}$ (except at $\infty$) or at 0. Next determine the Smith decomposition of Theorem \ref{T:Smith} for $P_2$:
\[
P_2(z) = E_2(z) D_2^-(z) D_2^\circ(z) D_2^+(z) F_2(z),
\]
where $E_2$ and $F_2$ are matrix polynomials with nonzero constant determinants, and $D_2^-$, $D_2^\circ$ and $D_2^+$ are diagonal matrix polynomials, with roots only inside $\BD$, on $\BT$ and outside $\overline{\BD}$, respectively, with in all three the diagonal entries monic and ordered as in Theorem \ref{T:Smith}.

Since 0 is the only root of $P_2$ in $\BD$, there exist $\rho_1\geq \rho_2 \geq \cdots \geq \rho_m\geq 0$ so that $D_2^-(z)=\diag(z^{\rho_1},\ldots,z^{\rho_m})$. Moreover, since $D_2^\circ$ is a diagonal matrix polynomial whose diagonal entries are monic polynomials with roots only on $\BT$, we can write $D_2^\circ(\frac{1}{z}) = \widetilde{D}_2^\circ(z) \widetilde{D}_2^-(\frac{1}{z})$ with $\wtil{D}_2^\circ,\wtil{D}_2^-\in\cP^{m\times m}$ diagonal matrix polynomials with $\wtil{D}_2^\circ$ having monic diagonal entries with roots only on $\BT$ and $\wtil{D}_2^-$ having roots only at zero. In fact, if $p_j$ is the $j$-th diagonal entry of $D_2^\circ$, then the $j$-th diagonal entry of $\wtil{D}_2^\circ$ is uniquely determined and given by $\frac{z^{\degr p_j}}{p_j(0)}p_j(\mbox{$\frac{1}{z}$})$, while the $j$-th diagonal entry of $\widetilde{D}_2^-(\frac{1}{z})$ is equal to $\frac{p_j(0)}{z^{\degr p_j}}$. In particular, $\wtil{D}_2^-(z)=\diag(p_1(0)^{-1}z^{\eta_1},\ldots,p_m(0)^{-1}z^{\eta_m})$ for integers $\eta_1\geq \eta_2\geq\cdots\geq\eta_m\geq 0$, since by construction $\degr p_j \geq \degr p_{j+1}$ for $j=1,\ldots,m-1$. We now obtain that
\begin{align}
  P_1(z)
& = z^N E_2(\mbox{$\frac{1}{z}$}) D_2^-(\mbox{$\frac{1}{z}$}) D_2^\circ(\mbox{$\frac{1}{z}$}) D_2^+(\mbox{$\frac{1}{z}$}) F_2(\mbox{$\frac{1}{z}$}) D_1^+(z)F_1(z) \notag\\
& = z^N E_2(\mbox{$\frac{1}{z}$})  D_2^+(\mbox{$\frac{1}{z}$}) D_2^\circ(\mbox{$\frac{1}{z}$}) D_2^-(\mbox{$\frac{1}{z}$}) F_2(\mbox{$\frac{1}{z}$}) D_1^+(z)F_1(z) \notag\\
& = z^N E_2(\mbox{$\frac{1}{z}$}) D_2^+(\mbox{$\frac{1}{z}$}) \wtil{D}_2^\circ(z)\wtil{D}_2^-(\mbox{$\frac{1}{z}$}) D_2^-(\mbox{$\frac{1}{z}$}) F_2(\mbox{$\frac{1}{z}$}) D_1^+(z)F_1(z).\label{P1fact1}
\end{align}
Since $E_2$ has a constant and nonzero determinant, for $z\neq 0$, $\det E_2(\frac{1}{z})$ is also constant and nonzero, so that $E_2(\frac{1}{z})\in\Rat^{m\times m}$ can only have zeroes and poles at 0. Hence $E_2(\frac{1}{z})$ is a minus function whose inverse is also a minus function. Furthermore, since $D_2^+(z)\in\cP^{m\times m}$ has only zeroes outside $\overline{\BD}$, the (diagonal) rational matrix function $D_2^+(\mbox{$\frac{1}{z}$})$ has zeroes only inside $\BD$ and can only have a pole at $0$. Thus $D_2^+(\mbox{$\frac{1}{z}$})$ is a minus function whose inverse is also a minus function. Hence the same conclusion holds for $E_2(\mbox{$\frac{1}{z}$}) D_2^+(\mbox{$\frac{1}{z}$})$.\smallskip


\paragraph{Step 4}
Let $K>0$ be the smallest integer such that
\[
P_3(z) := z^K \widetilde{D}_2^-(\mbox{$\frac{1}{z}$}) D_2^-(\mbox{$\frac{1}{z}$}) F_2(\mbox{$\frac{1}{z}$})\in\cP^{m\times m}.
\]
Note that $\det \widetilde{D}_2^-(z) D_2^-(\mbox{$\frac{1}{z}$}) F_2(\mbox{$\frac{1}{z}$})=c z^{-\xi}$ for some constant $c$ and with $\xi=\sum_{k=1}^m (\rho_k+ \eta_k)$, so that $\det P_3(z)=c z^{n}$, with $n:=m K-\xi$ being nonnegative by choice of $K$. Now apply Lemma \ref{L:fact_3} to $P_3$. It follows that we can write $P_3(z) = Q_3(z) F_3(z)$ with $Q_3,F_3\in\cP^{m\times m}$ with $F_3$ having a nonzero constant determinant and $Q_3$ lower triangular with $\det Q_3(z)=z^{n}$. Inserting this into \eqref{P1fact1} yields
\begin{align*}
  P_1(z)
& = z^{N-K} E_2(\mbox{$\frac{1}{z}$}) D_2^+(\mbox{$\frac{1}{z}$})\wtil{D}_2^\circ(z) Q_3(z) F_3(z)D_1^+(z)F_1(z).
\end{align*}
Using that $\Om(z)=q(z)^{-1}P_1(z)$ along with the factorization of $q^{-1}$ in Step 1, we obtain that
\begin{equation}\label{Omfact}
\Om(z)=z^{-k}\Om_-(z)\Om_\circ(z)P_0(z)\Om_+(z),
\end{equation}
where in case $N-K+\kappa\leq 0$ we set $k=-(N-K+\kappa)$ and define
\begin{equation}\label{Omfactors}
\begin{aligned}
\Om_-(z):= \om_-(z) E_2(\mbox{$\frac{1}{z}$}) D_2^+(\mbox{$\frac{1}{z}$}),&\quad
\Om_\circ(z):=\om_\circ(z)\wtil{D}_2^\circ(z),\\
P_0(z):=Q_3(z),
&\quad \Om_+(z):=\om_+(z)F_3(z)D_1^+(z)F_1(z),
\end{aligned}
\end{equation}
while if $N-K+\kappa> 0$ we set $k=0$, define $\Om_-$, $\Om_\circ$, $\Om_+$ as above and take $P_0(z)=z^{N-K+\kappa} Q_3(z)$.

Since both $\om_-(z)$ and $E_2(\mbox{$\frac{1}{z}$}) D_2^+(\mbox{$\frac{1}{z}$})$ are minus function whose inverses are minus functions, the same is true for $\Om_-(z)$. That $P_0(z)$ has the required form follows directly from Lemma \ref{L:fact_3}. It is also straightforward from the construction that $\Om_\circ(z)$ is a diagonal matrix whose diagonal entries are scalar rational functions with poles and zeroes only on $\BT$. Finally, all factors of $\Om_+(z)$ are plus functions whose inverses are also plus functions, hence $\Om_+(z)$ also has this property. We conclude that the factorization \eqref{Omfact}--\eqref{Omfactors} of $\Om$ has the required properties.
\end{proof}

\begin{proof}[\bf Proof of Theorem \ref{T:factUnique}]
Upon inspection of the proof of Theorem \ref{T:fact}, and specifically \eqref{Omfactors} using the definitions of $\om_-$, $\om_\circ$ and $\om_+$ from Step 1 of the proof, it follows that $P_+$ and $P_-$ in \eqref{P+DcircP-}, that is, $P_+(z)=F_3(z)D_1^+(z)F_1(z)$ and $P_-(z)=E_2(z)D_2^+(z)$,  are matrix polynomials with the required properties. It also follows that for $D_\circ$ in \eqref{P+DcircP-}, that is, $D_\circ(z)=\wtil{D}_2^\circ(z)$ as defined in Step~3 of the proof, the required ordering of the diagonal entries carries over from the corresponding ordering of $D_2^\circ$ from which $\wtil{D}_2^\circ$ is constructed. It is also clear from the construction, based on Lemma \ref{L:fact_3}, that we may take $P_0$ to have the required form. Finally, assume $k>0$, that is, $N-K+\kappa<0$. In case $P_0(0)=0$, we can write $P_0(z)=z^j\what{P}_0(z)$ for some $j>0$ and $\what{P}_0\in\cP^{m\times m}$ with $\what{P}_0(0)\neq 0$. In case $k\geq j$, replace $k$ by $k-j$ and $P_0$ by $\what{P}_0$ and in case $k<j$, replace $k$ by 0 and $P_0$ by $z^{-k}P_0(z)=z^{j-k}\what{P}_0(z)$. In both cases the adjusted $k$ and $P_0$ have the required relation, while the structure of $P_0$ is maintained. It follows that the factorization obtained from the construction in the proof of Theorem \ref{T:fact}, with the small modifications described here, has the required form.

Let \eqref{WHfact1} be a factorization of $\Om$ satisfying the conditions of Theorem \ref{T:fact} as well as the additional conditions of Theorem \ref{T:factUnique}. where $P_+$, $D_\circ$ and $P_-$ are defined as in \eqref{P+DcircP-}. Assume that a second factorization of this type exists:
\begin{equation}\label{SecFact}
\Om(z) = z^{-\wtil{k}}\wtil{\Om}_-(z) \wtil{\Om}_\circ(z) \wtil{P}_0(z) \wtil{\Om}_+(z),
\end{equation}
and write $\wtil{P}_+$, $\wtil{D}_\circ$ and $\wtil{P}_-$ for the matrix polynomials constructed via \eqref{P+DcircP-} for the factorization \eqref{SecFact}. It then follows that
\[
z^{-k} P_-(\mbox{$\frac{1}{z}$})D_\circ(z)P_0(z)P_+(z) = z^{-\wtil{k}} \wtil{P}_-(\mbox{$\frac{1}{z}$})\wtil{D}_\circ(z)\wtil{P}_0(z)\wtil{P}_+(z).
\]
Hence, for each integer $L\geq 0$ we have
\begin{align*}
& z^{\wtil{k}+L} \det(P_-(\mbox{$\frac{1}{z}$}))D_\circ(z)P_0(z)P_+(z)(\det(\wtil{P}_+(z)) \wtil{P}_+(z)^{-1})
=  \\
&\qquad = z^{k+L}(\det(P_-(\mbox{$\frac{1}{z}$})) P_-(\mbox{$\frac{1}{z}$})^{-1})
 \wtil{P}_-(\mbox{$\frac{1}{z}$})\wtil{D}_\circ(z)\wtil{P}_0(z)\det(\wtil{P}_+(z)).
\end{align*}
Now choose $L$ large enough so that both  $z^{k+L}(\det(P_-(\mbox{$\frac{1}{z}$})) P_-(\mbox{$\frac{1}{z}$})^{-1})
\wtil{P}_-(\mbox{$\frac{1}{z}$})$ and $z^{\wtil{k}+L} \det(P_-(\mbox{$\frac{1}{z}$}))$ are polynomials, so that on both sides of the identity we have a factorization of polynomials. Now observe that all four polynomials
\begin{align*}
z^{\wtil{k}+L} \det(P_-(\mbox{$\frac{1}{z}$})), &\quad  P_0(z)P_+(z)(\det(\wtil{P}_+(z)) \wtil{P}_+(z)^{-1}), \\
z^{k+L}(\det(P_-(\mbox{$\frac{1}{z}$})) P_-(\mbox{$\frac{1}{z}$})^{-1})
 \wtil{P}_-(\mbox{$\frac{1}{z}$}), &\quad \wtil{P}_0(z)\det(\wtil{P}_+(z)),
\end{align*}
do not have roots on $\BT$. Since the diagonal entries of $D_\circ$ and $\wtil{D}_\circ$ are ordered as stated in Theorem \ref{T:factUnique}, it now follows from the uniqueness claim of Theorem \ref{T:SmithReg} that both $D_\circ$ and $\wtil{D}_\circ$ are equal to the Smith form of the above polynomial with respect to $\BT$. In particular, $D_\circ$ and $\wtil{D}_\circ$ coincide.
%
%
\end{proof}

\begin{corollary}\label{C:SFP1}
Let $\Om\in\Rat^{m\times m}$ factor as \eqref{WHfact1} with the factors as in Theorem \ref{T:fact} satisfying the additional conditions of Theorem \ref{T:factUnique}. Let $q$ be the least common multiple of the denominators of the matrix entries of $\Om$. Then $D_\circ$ defined in \eqref{P+DcircP-} is equal to the Smith form of $P_1(z):= q(z)\Om(z)\in\cP^{m\times m}$ with respect to $\BT$.
\end{corollary}

\begin{proof}[\bf Proof]
By the uniqueness claim of Theorem \ref{T:factUnique} and the fact that the construction in the proof of Theorem \ref{T:fact} leads to a factorization that satisfies the additional conditions of Theorem \ref{T:factUnique}, it suffices to prove that the factor $D_\circ$ in \eqref{P+DcircP-} obtained from the construction of Theorem \ref{T:fact} coincides with the Smith form of $P_1$ with respect to $\BT$. In this case, $D_\circ(z)=\wtil{D}_2^\circ(z)$ as constructed in Step 3 of the proof of Theorem \ref{T:fact}, while $D_1^\circ$ from Step 2 of the proof is the Smith form of $P_1$ with respect to $\BT$. Hence we need to show that $D_1^\circ=\wtil{D}_2^\circ$.

We now follow the various steps of the construction of $\wtil{D}_2^\circ$ in the proof of Theorem \ref{T:fact}. Set $\wtil{P}_2(z):=E_1(z)D_1^-(z)D_1^\circ(z)\in\cP^{m\times m}$. Since $P_1(z)=\wtil{P}_2(z)D_1^+(z)F_1(z)$ and $D_1^+(z)F_1(z)\in\cP^{m\times m}$ is invertible for all $z\in\BT$, it follows from Corollary \ref{C:SmithMult} that $D_1^\circ$ is also the Smith form of $\wtil{P}_2$ with respect to $\BT$. Note that $P_2$ defined in Step 2 is given by $P_2(z)=z^{\degr \wtil{P}_2} \wtil{P}_2(\mbox{$\frac{1}{z}$})$, and that $D_2^\circ$ in Step 2 is the Smith form of $P_0$ with respect to $\BT$. The relation between $D_1^\circ$ and $D_2^\circ$ follows from Lemma \ref{L:invariantfact}. Say $D_1^\circ(z)=\diag(d_1^\circ(z),\ldots,d_m^\circ(z))$ and $D_2^\circ(z)=\diag(p_1^\circ(z),\ldots,p_m^\circ(z))$ with $d_j^\circ,p_j^\circ\in\cP$ monic and with roots only on $\BT$, for $j=1,\ldots,m$. Then $d_j^\circ(0)\neq 0$, $\degr d^\circ_j=\degr p^\circ_j$ and $p_j^\circ(z)=\frac{z^{\degr d^\circ_j}}{d^\circ_j(0)}d^\circ_j(\frac{1}{z})$. On the other hand, the $j$-th diagonal entry of $\wtil{D}_2^\circ$ is given by
\[
\frac{z^{\degr p^\circ_j}}{p^\circ_j(0)}p^\circ_j(\mbox{$\frac{1}{z}$})
=\frac{z^{\degr p^\circ_j}}{p^\circ_j(0)} \frac{z^{-\degr d^\circ_j}}{d^\circ_j(0)}d^\circ_j(z)
=\frac{1}{p^\circ_j(0)d^\circ_j(0)}d^\circ_j(z)=d^\circ_j(z),
\]
using $\degr p_j^\circ=\degr d_j^\circ$ in the first identity and the fact that both $d_j^\circ$ and the resulting polynomial are monic (so that $p^\circ_j(0)d^\circ_j(0)$ must be 1) in the second identity. Consequently, $d_j^\circ$ is the $j$-th diagonal entry of $\wtil{D}_2^\circ$, and thus $\wtil{D}_2^\circ=D_1^\circ$, as claimed.
\end{proof}

\begin{corollary}\label{C:fact1}
Let $\Om\in\Rat^{m\times m}$ with $\det \Om(z) \not\equiv 0$ and suppose
\[\Om(z) = z^{-k}\Om_-(z) \Om_\circ(z) P_0(z) \Om_+(z)\] is the factorization of $\Om$ as in Theorem \ref{T:fact}. Then the zeroes and poles of $\Om$ on $\BT$ correspond to the zeroes and poles of $\psi_j$ where $\Om_\circ(z) = \diag \left (\psi_j(z)\right )_{j=1}^m$.
\end{corollary}

%

\section{An example}\label{S:ExampleM}

In this section we present an example illustrating the factorization procedure. Consider
\[
\Om(z) =
\begin{pmatrix} 1 & \frac{1}{z-1} \\ 0 & 1 \end{pmatrix}.
\]
Then $q(z) = z-1$ and we have $q(z)^{-1}=(z-1)^{-1}=z^0 \om_-(z)\om_\circ(z)\om_+(z)$ with $\om_-(z)=\om_+(z)=1$ and $\om_\circ(z)=(z-1)^{-1}$. The Smith decomposition of $P_1(z)=q(z)\Om(z) = \sbm{ z-1 & 1 \\ 0 & z-1}$ is given by
\begin{align*}
P_1(z) & = E_1(z) D_1(z) F_1(z)
 = E_1(z) D_1^-(z) D_1^\circ(z) D_1^+(z) F_1(z)\\
& = \begin{pmatrix} 0 & 1 \\ -1 & z-1 \end{pmatrix}
\begin{pmatrix} (z-1)^2 & 0 \\ 0 & 1 \end{pmatrix}
\begin{pmatrix} 1 & 0 \\ z-1 & 1 \end{pmatrix}
\end{align*}
with $D_1^-(z) = D_1^+(z) = I_2 $. Then $E_1(z) D_1^-(z) D_1^\circ(z)\! =\! \sbm{0 & 1 \\ -(z-1)^2 & z-1}$, and so
\begin{align*}
P_2(z) &= z^N E_1(\mbox{$\frac{1}{z}$}) D_1^-(\mbox{$\frac{1}{z}$}) D_1^\circ(\mbox{$\frac{1}{z}$})\\
&= z^2 \begin{pmatrix} 0 & 1 \\ -(\frac{1}{z} - 1)^2 & \frac{1}{z} - 1 \end{pmatrix} =
\begin{pmatrix} 0 & z^2 \\ -(z - 1)^2 & z - z^2 \end{pmatrix}.
\end{align*}
The Smith decomposition of $P_2$ is given by
\begin{align*}
&P_2(z)  =  E_2(z) D_2(z) F_2(z) \\
&\quad  =
\begin{pmatrix} -z-1 & -z^2 \\ 1 & z-1 \end{pmatrix}
\begin{pmatrix} (z-1)^2 z^2 & 0 \\ 0 & 1 \end{pmatrix}
\begin{pmatrix} -1 & -1 \\ (z-1)^2 (z+1) & z(z^2 - z -1)\end{pmatrix}.
\end{align*}
Hence
\[
D_2^-(z) = \begin{pmatrix} z^2 & 0 \\ 0 & 1 \end{pmatrix}, \quad
D_2^\circ(z) = \begin{pmatrix} (z-1)^2 & 0 \\ 0 & 1\end{pmatrix}\quad \textrm{and}\quad D_2^+(z) = I_2.
\]
Put
\[
D_2^\circ(\mbox{$\frac{1}{z}$})  =
\begin{pmatrix} \left(\frac{1}{z} -1 \right)^2 & 0 \\ 0 & 1 \end{pmatrix}
= \begin{pmatrix} (1-z)^2 & 0 \\ 0 & 1 \end{pmatrix}
\begin{pmatrix} z^{-2} & 0 \\ 0 & 1 \end{pmatrix}=\widetilde{D}_2^\circ(z) \widetilde{D}_2^-(\mbox{$\frac{1}{z}$}).
\]
Then
\begin{align*}
P_3(z) & =  z^K \widetilde{D}_2^-(z) D_2^- (\mbox{$\frac{1}{z}$}) F_2(\mbox{$\frac{1}{z}$}) \\
& =
z^4 \begin{pmatrix} z^{-2} & 0 \\ 0 & 1\end{pmatrix}
\begin{pmatrix} z^{-2} & 0 \\ 0 & 1\end{pmatrix}
\begin{pmatrix} -1 & -1 \\ \left(\frac{1}{z} - 1\right)^2 \left(\frac{1}{z} + 1\right ) & \frac{1}{z} \left( \frac{1}{z^2} - \frac{1}{z} -1 \right ) \end{pmatrix},
\end{align*}
from which it follows that
\[
P_3(z) = \begin{pmatrix} -1 & -1 \\ z(1 - z)^2(1+z) & z(1 - z - z^2) \end{pmatrix}\quad\mbox{ and thus } \det P_3(z) = z^4.
\]
The above computations conclude Steps 1, 2 and 3 of the procedure in the proof of Theorem \ref{T:fact}. To conclude Step 4 we have to apply the recursive procedure from the proof of Lemma 3.2 to factor $P_3$.

In the first step we factor
\[
P_3(z) = Q_1(z) R_1(z) = \begin{pmatrix} 1 & 0 \\ 0 & z\end{pmatrix}
\begin{pmatrix} -1 & -1 \\ z^3 - z^2 - z + 1 & 1 - z - z^2 \end{pmatrix}.
\]
The sum of the diagonal multiplicities is $1$ which is less than $4$. Note that $r_1(0)=(-1 \ -1)$ and $r_2(0) = (1 \ 1)$. Hence we find $L_1 =\sbm{ 1 & 0 \\ 1 & 1}$. Then
\begin{align*}
P_3(z) & = Q_1(z) L_1^{-1} L_1 R_1(z)  \\
& =   \begin{pmatrix} 1 & 0 \\ 0 & z\end{pmatrix}
\begin{pmatrix} 1 & 0 \\ -1 & 1 \end{pmatrix}
\begin{pmatrix} 1 & 0 \\ 1 & 1\end{pmatrix}
\begin{pmatrix} -1 & -1 \\ z^3 - z^2 - z + 1 & 1 - z - z^2 \end{pmatrix} \\
& = \begin{pmatrix} 1 & 0 \\ 0 & z\end{pmatrix}
\begin{pmatrix} 1 & 0 \\ -1 & 1 \end{pmatrix}
\begin{pmatrix} -1 & -1 \\ z^3 - z^2 - z  &  - z - z^2 \end{pmatrix} \\
& =
\begin{pmatrix} 1 & 0 \\ -z & 1 \end{pmatrix}
\begin{pmatrix} 1 & 0 \\ 0 & z\end{pmatrix}
\begin{pmatrix} 1 & 0 \\ 0 & z\end{pmatrix}
\begin{pmatrix} -1 & -1 \\ z^2 - z - 1 &  - 1 - z \end{pmatrix}\\
& = \begin{pmatrix} 1 & 0 \\ -z & z^2 \end{pmatrix}
\begin{pmatrix} -1 & -1 \\ z^2 - z - 1 &  - 1 - z \end{pmatrix} = Q_2(z) R_2(z).
\end{align*}
The sum of the diagonal multiplicities is $2$, less than $4$. The rows of $R_2(0)$ are $r_1'(0)= r_2'(0)= (-1 \ -1)$. Hence we take $L_2 =\sbm{1 & 0 \\ -1 & 1 }$. Then
\begin{align*}
P_3(z) & = Q_2(z) R_2(z) = Q_2(z) L_2^{-1} L_2 R_2(z) \\
& = \begin{pmatrix} 1 & 0 \\ -z & z^2 \end{pmatrix}
\begin{pmatrix} 1 & 0 \\ 1 & 1\end{pmatrix}
\begin{pmatrix} 1 & 0 \\ -1 & 1\end{pmatrix}
\begin{pmatrix} -1 & -1 \\ z^2 - z - 1 &  - 1 - z \end{pmatrix}\\
& = \begin{pmatrix} 1 & 0 \\ -z & z^2 \end{pmatrix}
\begin{pmatrix} -1 & -1 \\ z^2 - z  &   - z \end{pmatrix}
= \begin{pmatrix} 1 & 0 \\ -z & z^2 \end{pmatrix}  \begin{pmatrix} 1 & 0 \\ 0 & z\end{pmatrix}
\begin{pmatrix} -1 & -1 \\ z - 1  &   -1 \end{pmatrix}\\
& = \begin{pmatrix} 1 & 0 \\ z^2 -z & z^3 \end{pmatrix}
\begin{pmatrix} -1 & -1 \\ z - 1  &   -1 \end{pmatrix} = Q_3(z) R_3(z).
\end{align*}
The sum of the diagonal multiplicities is 3, still less than 4, and so we apply the procedure one more time. Now $r_1''(0)=r_2''(0) ( -1\ -1)$ and so we have $L_3 =\sbm{1 & 0 \\ -1 & 1}$. Then
\begin{align*}
P_3(z) & = Q_3(z) R_3(z) =Q_3(z)L_3^{-1} L_3 R_3(z)\\
& = \begin{pmatrix} 1 & 0 \\ z^2 -z & z^3 \end{pmatrix}
\begin{pmatrix} 1 & 0 \\ 1 & 1 \end{pmatrix}
\begin{pmatrix} 1 & 0 \\ -1 & 1\end{pmatrix}
\begin{pmatrix} -1 & -1 \\ z - 1  &   -1 \end{pmatrix}\\
& = \begin{pmatrix} 1 & 0 \\ z^3 + z^2 -z & z^3 \end{pmatrix}
\begin{pmatrix} -1 & -1 \\ z   &   0 \end{pmatrix}\\
& = \begin{pmatrix} 1 & 0 \\ z^3 + z^2 -z & z^4 \end{pmatrix}
\begin{pmatrix} -1 & -1 \\  1  &   0 \end{pmatrix} = Q_4(z) R_4(z).
\end{align*}
This provides the factorization of $P_3$ from Lemma \ref{L:fact_3}. We have now computed all required matrix functions for the factorization of $\Om$ in \eqref{Omfact}--\eqref{Omfactors}. This yields
\[
\Om(z)=z^{N-K+\kappa}\Om_-(z)\Om_\circ(z)P_0(z)\Om_+(z),
\]
with $N-K+\kappa=2-4+0=-2$ and where
\begin{align*}
\Om_-(z) & = \om_-(z) E_2(\mbox{$\frac{1}{z}$}) D_2^+(\mbox{$\frac{1}{z}$})=E_2(\mbox{$\frac{1}{z}$})
=\begin{pmatrix} -\frac{1}{z}-1 & -\frac{1}{z^2} \\ 1 & \frac{1}{z}-1 \end{pmatrix},\\
\Om_\circ(z) & = \om_\circ(z) \wtil{D}_2^\circ(z)=\frac{1}{z-1} \begin{pmatrix} (1-z)^2 & 0 \\ 0 & 1 \end{pmatrix}
=\begin{pmatrix} z-1 & 0 \\ 0 & \frac{1}{z-1} \end{pmatrix},\\
\Om_+(z) &=\om_+(z)R_4(z)D_1^+(z)F_1(z)=R_4(z)F_1(z)=\begin{pmatrix} -z & -1 \\ 1 & 0 \end{pmatrix},\\
&\quad \mbox{and}\quad P_0(z)=Q_4(z)=\begin{pmatrix} 1 & 0 \\ z^3 + z^2 -z & z^4 \end{pmatrix}.
\end{align*}

\section{Factorization of the Toeplitz operator}\label{S:FactTM}

In this section we prove Theorem \ref{T:Fact_Toeplitz} by using the Wiener-Hopf type factorization of Theorem \ref{T:fact}. We first prove some technical lemmas.

\begin{lemma}\label{L:WH_a}
Suppose that $\Om,U,V\in\Rat^{m\times m}$ with $U$ a minus function whose inverse is a minus function and $V$ a plus fuction. Then $T_{\Om V} = T_\Om T_V$ and $T_{U\Om} = T_U T_\Om$.
\end{lemma}

\begin{proof}[\bf Proof]
Let $f\in \Dom(T_{\Om V})$. Then $\Om Vf=h+\rho$ for $h\in L_m^p $ and $\rho\in \Rat_0^m(\BT)$ and $T_{\Om V}f=\BP h$. Since $V$ has no poles in $\overline{\BD}$, $V$ is analytic and bounded on $\BD$. Therefore, $T_V$ is bounded and $T_V f= Vf$. Put $g=T_V f=Vf$. Then by the above relation we have $g\in \Dom(T_\Om)$ and $T_\Om g=\BP h$. In other words, $T_\Om T_Vf=T_{\Om V}f$. So $T_V$ maps $\Dom(T_{\Om V})$ into $\Dom(T_\Om)$ and on $\Dom (T_{\Om V})$ the operators $T_{\Om V}$ and $T_\Om T_V$ coincide. To see that $T_{\Om V}=T_\Om T_V$, suppose that $u\in \Dom(T_\Om T_V)$. Since $V$ is bounded and analytic on $\BD$, we have $Vu=T_Vu\in \Dom(T_\Om)$, that is, $\Om Vu=w+\eta$ for some $w \in L_m^p $ and $\eta \in\Rat_0^m(\BT)$, and so $u\in \Dom(T_{\Om V})$. This proves that
$\Dom(T_{\Om} T_V)=\Dom(T_{\Om V})$, and hence that $T_{\Om} T_V = T_{\Om V}$.

Next we prove that $T_{U\Om} = T_U T_\Om$. Let $f\in \Dom(T_{U\Om})$. Then $U\Om f=h+\rho$ with $h\in L_m^p$ and $\rho\in\Rat_0(\BT)$. Hence $\Om f=U^{-1}h+U^{-1}\rho$. Since $U^{-1}$ is minus function, it is analytic outside $\overline{\BD}$ and thus $U^{-1}\rho$ can be written as $h_1+\rho_1$ with $h_1$ a rational vector function in $L_m^p$ with poles only in the unit disc and $\rho_1\in\Rat_0^m(\BT)$. In particular, $\BP h_1 = 0$. So $\Om f=U^{-1}h+h_1+\rho_1$, which shows that $f\in \Dom(T_\Om)$ and $T_\Om f=\BP(U^{-1}h+h_1)=\BP(U^{-1}h)$. Since $U$ is a minus function, we have $\BP U(I-\BP) (U^{-1}h)=0$. Therefore, we find that
\begin{align*}
T_UT_\Om f
&=T_U \BP(U^{-1}h)=\BP U\BP(U^{-1}h)
=\BP U\BP(U^{-1}h)+\BP U(I-\BP)(U^{-1}h)\\
&=\BP UU^{-1}h=\BP h=T_{U\Om}f.
\end{align*}
We proved that $\Dom (T_{U\Om})\subset \Dom(T_\Om)=\Dom(T_U T_\Om)$ and that $T_{U \Om}$ and $T_U T_V$ coincide on $\Dom (T_{U\Om})$. It remains to prove $\Dom(T_U T_\Om)\subset \Dom (T_{U\Om})$. Let $v\in \Dom(T_U T_\Om)= \Dom(T_\Om)$. Then $\Om v= w+\eta$ for $w\in L^p_m$ and $\eta\in\Rat_0^m(\BT)$. Then $U\Om v= Uw+U\eta$ and because $U$ is a minus function, $Uw\in L^p_m$ and $U\eta= w'+\eta'$ for $w'\in L^p_m$ and $\eta'\in \Rat_0^m(\BT)$. Hence $U\Om v = U w+ w'+ \eta '\in L^p_m + \Rat_0^m(\BT)$, so that $v\in\Dom(T_{U\Om})$.
\end{proof}

\begin{lemma}\label{L:z_minus_k}
Let $\Om\in\Rat^{m\times m}(\BT)$. Then for $k \geq 0$, $T_{z^{-k}\Om} = T_{z^{-k}I_k} T_\Om$.
\end{lemma}

\begin{proof}[\bf Proof]
It suffices to show that $\Dom (T_\Om) = \Dom (T_{z^{-k}\Om})$ and that $T_{z^{-k}\Om}$ and $T_{z^{-k}I_m}T_\Om$ coincide on $\Dom (T_\Om)$. Suppose that  $f\in\Dom (T_{z^{-k}\Om})$. Then $z^{-k}\Om f = h + \rho$ with $h\in L_m^p$ and $\rho\in\Rat_0^m(\BT)$. Then $\Om f = z^k h + z^k \rho$ and $z^k h$ is still in $L_m^p$. Apply the Euclidian algorithm entrywise to write $z^k \rho = \rho_1 + \rho_2$ with $\rho_2\in \Rat_0^m(\BT)$ and $\rho_1\in\cP_{k-1}^m$. Then $\Om f = (z^k h + \rho_1) + \rho_2 \in L_m^p+\Rat_0^m(\BT)$ from which it follows that $f\in\Dom(T_\Om)$. Let $r\in\cP_{k-1}^m$ so that $\BP z^k h = z^k \BP h +r$. Then we have
\begin{align*}
T_{z^{-k}I_m} T_\Om f & = T_{z^{-k}I_m} \BP (z^k h +\rho_1)= T_{z^{-k}I_m} (z^k \BP h +r +\rho_1)\\
&= T_{z^{-k}I_m} z^k \BP h =\BP h = T_{z^{-k}\Om} f.
\end{align*}
Thus $\Dom (T_{z^{-k}\Om})\subset \Dom (T_\Om)$ and the operators $T_{z^{-k}I_m} T_\Om$ and $T_{z^{-k}\Om}$ coincide on $\Dom (T_{z^{-k}\Om})$. To prove the converse inclusion, suppose $u\in \Dom (T_\Om)$. Then $\Om u = w + \eta$ for $w\in L_m^p $ and $\eta\in \Rat_0^m(\BT)$, so that $z^{-k}\Om u = z^{-k}w + z^{-k}\eta \in L_m^p + \Rat_0^m(\BT \cup \{0\})$. Now write $z^{-k}\eta = \eta_1 + \eta_2$ with $\eta_2\in\Rat_0^m(\BT)$ and $\eta_1\in\Rat^m$, with only a pole at 0. Then $\eta_1\in L_m^p$ and $\BP \eta_1 = 0$. We now have $z^{-k} \Om u  = (z^{-k} w + \eta_1) + \eta_2 \in L_m^p + \Rat_0^m(\BT)$ and so $u\in \Dom (T_{z^{-k}\Om})$.
\end{proof}

\begin{proof}[\textbf{Proof of Theorem \ref{T:Fact_Toeplitz}}]
Factor $\Om = z^{-k}\Om_- \Om_\circ P_0\Om_+$ as in Theorem \ref{T:fact}. Then $P_0 \Om_+$ is a plus function, and $\Om_-$ is a minus function whose inverse is also a minus function. Thus, by Lemma \ref{L:WH_a}, we can write $T_\Om = T_{\Om_-} T_{z^{-k}\Om_\circ} T_{P_0} T_{\Om_+}$.
Applying Lemma \ref{L:z_minus_k} we have
$
T_\Om = T_{\Om_-} T_{z^{-k}I_m} T_{\Om_\circ} T_{P_0} T_{\Om_+}.
$
\end{proof}

\section{Fredholm properties}\label{S:FredholmM}

Using the Wiener-Hopf type factorization from Theorem \ref{T:fact} and the corresponding factorization of Toeplitz operators in Theorem \ref{T:Fact_Toeplitz} we are now in a position to prove Theorem \ref{T:Fredholm} via a partial reduction to the diagonal case.

\begin{proof}[\textbf{Proof of Theorem \ref{T:Fredholm}}]
Let $\Om\in\Rat^{m\times m}$ be factored as in Theorem \ref{T:fact}:
\[
\Om(z)=z^{-k}\Om_-(z) \Om_\circ(z) P_0(z)\Om_+(z),
\]
with $k\geq 0$, $\Om_+,\Om_-,\Om_\circ\in\Rat^{m\times m}$ so that $\Om_-$ a minus function whose inverse is a minus function, $\Om_+$ a plus function whose inverse is a plus function, $\Om_\circ=\diag(\phi_1,\ldots,\phi_m)$ a diagonal matrix function whose entries have poles and zeroes only on $\BT$ and $P_0$ a lower triangular polynomial matrix with $\det P_0(z)=z^n$ for some integer $n\geq 0$. Applying Theorem \ref{T:Fact_Toeplitz} gives
\[
T_\Om = T_{\Om_-} T_{z^{-k}} T_{\Om_\circ}T_{P_0} T_{\Om_+}.
\]
Given that $\Om_-$ and its inverse are minus functions, $T_{\Om_-}$ is invertible with $T_{\Om_-}^{-1} = T_{\Om_-^{-1}}$. Similarly, $T_{\Om_+}$ is invertible and $T_{\Om_+}^{-1} = T_{\Om_+^{-1}}$. Thus $T_{\Om_-}$,  $T_{\Om_-^{-1}}$, $T_{\Om_+}$ and $T_{\Om_+^{-1}}$ are all Fredholm with index 0 and
\[
T_{z^{-k}I_m} T_{\Om_\circ}T_{P_0} = T_{\Om_-^{-1}} T_\Om T_{\Om_+^{-1}},\quad T_\Om   = T_{\Om_-} T_{z^{-k}I_m} T_{\Om_\circ}T_{P_0} T_{\Om_+}.
\]
Applying item (iii) of Theorem IV.2.7 from \cite{G66} (see also \cite{GK57}) it now follows that $T_\Om$ is Fredholm if and only if $T_{z^{-k}I_m} T_{\Om_\circ}T_{P_0}$ is Fredholm and in that case we have
\[
\Index (T_\Om) = \Index (T_{z^{-k}I_m} T_{\Om_\circ}T_{P_0}).
\]

Since $P_0(z)$ is a matrix polynomial with zeroes only at $0$, $T_{P_0}$ is a bounded Fredholm operator. Therefore, if $T_{z^{-k}} T_{\Om_\circ}T_{P_0}$ is Fredholm, then we find that $T_{z^{-k}} T_{\Om_\circ}$ is Fredholm, by Theorem 3.4 of \cite{S67}, while conversely, if $T_{z^{-k}I_m} T_{\Om_\circ}$ is Fredholm, then $T_{z^{-k}I_m} T_{\Om_\circ}T_{P_0}$ by Theorem IV.2.7 from \cite{G66}. Furthermore, in this case we have
\[
\Index (T_{z^{-k}I_m} T_{\Om_\circ}T_{P_0})=\Index(T_{z^{-k}I_m} T_{\Om_\circ}) + \Index(T_{P_0}).
\]
Moreover, consider the Wiener-Hopf factorization $P_0(z)=\Phi_-(z) D_0(z) \Phi_+(z)$ of $P_0$, cf., Theorem XXIV.3.1 in \cite{GGK2}, with $\Phi_-$ and $\Phi_-^{-1}$ minus functions, $\Phi_+$ and $\Phi_+^{-1}$ plus functions, and $D_0(z)=\diag(z^{k_1},\ldots,z^{k_m})$ for integers $k_1,\ldots,k_m$. Since $\det P_0(z)=z^n$ with $n=\sum_{j=1}^m n_j$, the sum of exponents of $z$ on the diagonal of $P_0$, we have $\det D_0(z)=z^n$ so that $-\sum_{j=1}^m n_j=-n=-\sum_{j=1}^m k_j$ is the Fredholm index of $T_{P_0}$.

Since $\Om_\circ=\diag(\phi_1,\ldots,\phi_m)$ and $z^{-k}I_m$ are diagonal matrices, the question on Fredholm properties of $T_{z^{-k}\Om_\circ}$ viz-a-viz $T_{z^{-k}I_m}T_{\Om_\circ}$, reduces to the scalar case for each diagonal entry $z^{-k}\phi_j$, $j=1,\ldots,m$. It follows from Proposition \ref{P:diagonal} that $T_{z^{-k}\Om_\circ}$ is Fredholm if and only if all operators $T_{z^{-k}\phi_j}$, $j=1,\ldots,m$, are Fredholm, and in this case
\[
\Index(T_{z^{-k}\Om_\circ})=\sum_{j=1}^m \Index(T_{z^{-k}\phi_j}).
\]
We can thus invoke Theorem 1.1 from \cite{GtHJR1} to conclude that $T_{z^{-k}\Om_\circ}$ is Fredholm, or equivalently, $T_\Om$ is Fredholm, if and only if non of the entries $\phi_j$ of $\Om_0$ has a zero on $\BT$ (equivalently, $\det\Om_\circ (z)$ has no zeroes on $\BT$). Furthermore, in case $T_{z^{-k}\Om_\circ}$ is Fredholm we have $\phi_j=1/q_j$ for some $q_j\in\cP$ and
\[
\Index(T_{z^{-k}\Om_\circ}) = \sum_{j=1}^m k+ \degr q_j=mk + \sum_{j=1}^m \degr q_j.
\]
Combining the above observations we obtain that $T_\Om$ is Fredholm if and only if $\det \Om_0(z)$ has no zeroes on $\BT$, and in case $T_\Om$ is Fredholm we have
\begin{align*}
\Index(T_\Om)
&=\Index (T_{z^{-k}I_m} T_{\Om_\circ}T_{P_0})=\Index(T_{z^{-k}I_m} T_{\Om_\circ}) + \Index(T_{P_0})\\
&= mk + \sum_{j=1}^m \degr q_j - \sum_{j=1}^m n_j,
\end{align*}
where $n_1,\ldots,n_m\geq 0$ are the exponents of $z$ on the diagonal entries of $P_0$ and $q_1,\ldots,q_m$ are the denominators of the co-prime representations of the diagonal entries of $\Om_\circ$. This completes the proof of Theorem \ref{T:Fredholm}.
\end{proof}

\begin{remark}
It now follows from the decomposition of $\Om(z) = \sbm{ 1 & \frac{1}{z-1}\\ 0 & 1 }$ in Section \ref{S:ExampleM} that $T_\Om$ is not Fredholm, despite $\det \Om(z)$ not having a zero on $\BT$.
\end{remark}

From the proof of Theorem \ref{T:Fredholm} it follows that $T_\Om$ is Fredholm if and only if $T_\Xi$ is Fredholm, where
\[
\Xi(z)=z^{-k}\Om_\circ(z)P_0(z)
\]
with $k$, $\Om_\circ$ and $P_0$ from any Wiener-Hopf type factorization of $\Om$ as in Theorem \ref{T:fact}. Moreover, if $T_\Om$ is Fredholm, then $\Index(T_\Om)=\Index(T_\Xi)$, in fact, one has $\dim \kernel T_\Om = \dim \kernel T_\Xi$ and $\codim \Ran T_\Om = \codim \Ran T_\Xi$. Assume $T_\Om$ is Fredholm, so that $\Om_\circ$ has the form $\Om_\circ(z)=\diag(1/q_1,\ldots,1/q_m)$ with $q_j$ a divisor of $q_{j+1}$ for $j=1,\ldots,m-1$. Let the diagonal entries of $P_0$ be $z^{n_1},\ldots,z^{n_m}$ and let $p_{i,j}\in\cP$ be the lower triangular off-diagonal entry in position $(i,j)$, for $i>j$. By construction $\degr p_{i,j}<n_i$. Then $\Xi$ has the form
\begin{equation}\label{MatRepXi}
 \Xi(z)=\mat{\frac{z^{n_1}}{z^k q_1(z)}&0&\cdots&0\\ \frac{p_{2,1}(z)}{z^k q_2(z)}&\frac{z^{n_2}}{z^k q_2(z)}&\ddots&\vdots\\ \vdots&\ddots&\ddots&0\\ \frac{p_{m,1}(z)}{z^k q_m(z)}&\cdots&\frac{p_{m,m-1}(z)}{z^k q_m(z)}&\frac{z^{n_m}}{z^k q_m(z)}}.
\end{equation}
One may wonder to what extent this form is unique. For instance, is it maybe the case that the numbers $n_1,\ldots,n_m$ are unique? The following examples show that this is not the case.

\begin{example}\label{E:notUnique1}
Let
\[
\Xi(z)=\mat{\frac{1}{(z-1)^3} & 0 \\ \frac{z^2}{(z-1)^4} & \frac{z^5}{(z-1)^4}}.
\]
This function is of the form $\Xi_0(z)P_0(z)$, with $\Xi_0(z)={\rm diag\,} \left(\frac{1}{(z-1)^3}  , \frac{1}{(z-1)^4}\right)$ and $P_0(z)=\mat{1 & 0 \\ z^2 & z^5 }$. So for this factorization we have that $n_1=0, n_2=5$, and $k=0$.

Introduce
\[
\Xi_-(z)=\mat{ -1 & \frac{z-1}{z^2}\\ 0 & 1}, \qquad \Xi_+(z)=\mat{-z^3 & 1 \\ 1 & 0},
\]
then $\Xi_-$ and its inverse are minus functions and $\Xi_+$ and its inverse are plus functions,
and
\[
\Xi(z)=\Xi_-(z)^{-1}\mat{ \frac{z^3}{(z-1)^3} & 0 \\ 0 &\frac{z^2}{(z-1)^4}  }\Xi_+(z)^{-1},
\]
which is more like an ordinary Wiener-Hopf factorization. Note that the middle factor is also of the form as in Theorem 1.2, but now with $n_1=3$ and $n_2=2$. In line with Theorem \ref{T:factUnique}, the denominators on the diagonal, $(z-1)^3$ and $(z-1)^4$, do not change.

Both factorizations tell us that $T_\Xi $ is Fredholm with index $2$.
It is obvious we would prefer the second factorization, as in that factorization the degrees of $q_1$ and $q_2$ and the $n_1$ and $n_2$ give us information on the dimension of the kernel and codimension of the range of $T_{\Xi}$ using Proposition \ref{P:diagonal}. In fact, it can be checked directly that the dimension of the kernel of $T_\Xi$ is two and since $T_\Xi$ is Fredholm with index two it follows that $T_\Xi$ is onto.
\end{example}

The example raises the question whether it might always be possible to diagonalize the middle term \eqref{MatRepXi} by multiplying on the left with a minus function that has a minus function inverse and on the right with a plus function that has a plus function inverse. The following example shows a more general $2 \times 2$ case, where the procedure from Example \ref{E:notUnique1} can be carried out.

\begin{example}\label{E:notUnique2}
Start by considering
\[
\Omega(z)=
\mat{\frac{z^{k_1}}{q_1(z) }& 0 \\ \frac{d(z)}{q_2(z) }& \frac{z^{k_2}}{q_2(z)}}
\]
where $q_1$ a divisor of $q_2$ and $q_2$ has all its roots on $\mathbb{T}$, and ${\rm deg\, }d <k_2$.
Write $d(z)=z^{k_{12}}d_0(z)$, with $d_0(0)\not= 0$, and $k_{12}+{\rm deg\, }d_0 <k_2$. Then $d_0(z)$ and $z^{k_2-k_{12}}$ have greatest common divisor $1$, and so by the Bezout identity there are polynomials $p_1(z)$ and $p_2(z)$ such that
$d_0(z)p_1(z)+z^{k_2-k_{12}}p_2(z)=1$, and the degree of $p_1(z)$ is less than or equal to $k_2-k_{12}-1$, while the degree of $p_2(z)$ is less than or equal to the degree of $d_0$ minus one.
Set
\[
\Omega_+(z)=\mat{ -z^{k_2-k_{12}} & p_1(z) \\ d_0(z) & p_2(z)}.
\]
Then $\Omega_+$ is a plus function, and since the determinant of $\Omega_+$ is one by the Bezout identity, also the inverse of $\Omega_+$ is a plus function.
Now
\[
\Omega(z)\Omega_+(z) = \mat{
-\frac{z^{k_1+k_2-k_{12}}}{q_1(z)}  & \frac{z^{k_1}p_1(z)}{q_1(z)}\\
0 & \frac{z^{k_{12}}}{q_2(z)}}.
\]
Write $q_2(z)=q_0(z)q_1(z)$. Now assume that
\begin{equation}\label{condition}
k_1+{\rm deg\, } q_0+{\rm deg\,} p_1 \leq k_{12}.
\end{equation}
Set
\[
\Omega_-(z)=\mat{
-1 & \frac{z^{k_1}p_1(z)q_0(z)}{z^{k_{12}}}\\
0 & 1}
\]
which is a minus function with an inverse that is also a minus function. Then
\[
\Omega_-(z) \Omega(z)\Omega_+(z) = \mat{
\frac{z^{k_1+k_2-k_{12}}}{q_1(z)}& 0 \\ 0 &  \frac{z^{k_{12}}}{q_2(z)}} .
\]
Note that in the previous example condition \eqref{condition} is satisfied. Without \eqref{condition}, however, $\Omega_-$ above would have a pole at $\infty$ and hence would not be a minus function.
\end{example}

\paragraph{\bf Data availability statement}
Data sharing is not applicable to this article as no datasets were generated or analysed during the current study.

\paragraph{\bf Acknowledgments}
This work is based on research supported in part by the National Research Foundation of South Africa (NRF) and the DSI-NRF Centre of Excellence in Mathematical and Statistical Sciences (CoE-MaSS). Any opinion, finding and conclusion or recommendation expressed in this material is that of the authors and the NRF and CoE-MaSS do not accept any liability in this regard.


\begin{thebibliography}{99}

\bibitem{B68}
V. Belevitch, {\em Classical network theory}, San Francisco-Cambridge-Amsterdam, Holden Day, 1968.

\bibitem{CG}
K.F. Clancey and I. Gohberg, {\em Factorization of Matrix Functions and Singular Integral Operators}, Oper.\ Theory Adv.\ Appl.\  {\bf 3}, Birkh\"{a}user Verlag, Basel, 1981.

\bibitem{DS74}
C.A. Desoer and J.D. Schulman, Zeros and poles of matrix transfer functions and their dynamical interpretation, {\em IEEE Trans.\ Circ.\ and Systems} {\bf 21} (1974), 3--8.

\bibitem{G59}
F.R. Gantmacher, {\em The theory of matrices}, 2nd ed., Nauka, Moscow, 1966, English transl.\ of 1st ed., Chelsea, New York, 1959.

\bibitem{G66}
S. Goldberg, {\em Unbounded linear operators. Theory and Applications,} Dover Publications, New York, 1966.

\bibitem{GGK92a}
I. Gohberg, S. Goldberg and M.A. Kaashoek, {\em Classes of Linear Operators. Vol.\ 1}, Oper.\ Theory Adv.\ Appl.\ {\bf  49}, Birkh\"{a}user Verlag, Basel, 1990.

\bibitem{GGK2}
I. Gohberg, S. Goldberg and M.A. Kaashoek, {\em Classes of Linear Operators. Vol.\ 2}, Oper.\ Theory Adv.\ Appl.\  {\bf 63}, Birkh\"{a}user Verlag, Basel,  1993.

\bibitem{GK57}
I.C. Gohberg and M.G. Kre\u{\i}n, Fundamental aspects of defect numbers, root numbers and indexes of linear operators (Russian), {\em Uspehi Mat.\ Nauk (N.S.)} {\bf 12.2} (74) (1957), 43--118. English transl., {\em Am.\ Math.\ Soc.\ Transls.} ser.\ 2, vol.\ {\bf 13}, 1960.

\bibitem{GK58}
I.C. Gohberg and M.G. Kre\u{\i}n, Systems of integral equations on a half-line with kernels depending on the difference of the arguments, {\em Uspehi Mat.\ Nauk} {\bf 13.2} (80) (1958), 3--72. English transl., {\em Am.\ Math.\ Soc.\ Transls.}  vol.\ {\bf 13}, 1960.


\bibitem{GLR82}
I. Gohberg, P. Lancaster, and L. Rodman, {\em Matrix polynomials}, Computer Science and Applied Mathematics, Academic Press, Inc., New York-London, 1982.

\bibitem{GtHJR1}
G.J. Groenewald, S. ter Horst, J. Jaftha and A.C.M. Ran, A Toeplitz-like operator with rational symbol having poles on the unit circle I: Fredholm properties, {\em Oper.\ Theory Adv.\ Appl.} {\bf 271} (2018), 239--268.

\bibitem{GtHJR2}
G.J. Groenewald, S. ter Horst, J. Jaftha and A.C.M. Ran, A Toeplitz-like operator with rational symbol having poles on the unit circle II: The spectrum, {\em Oper.\ Theory Adv.\ Appl.} {\bf 272} (2019), 133--154.

\bibitem{GtHJR3}
G.J. Groenewald, S. ter Horst, J. Jaftha and A.C.M. Ran, A Toeplitz-like operator with rational symbol having poles on the unit circle III: The adjoint, {\em Integr.\ Equ.\ Oper.\ Theory} {\bf 91} (2019), article number 43, 23 pages.

\bibitem{H62}
K. Hoffman, {\em Banach spaces of analytic functions}, Prentice-Hall Inc, New Jersey, 1962.

\bibitem{HJ85}
R.A. Horn and C.R. Johnson, {\em Matrix analysis} ,Cambridge University Press, Cambridge, 1985.

\bibitem{P08}
J. Plemelj, Riemannsche Funktionenscharen mit gegebener Monodromiegruppe, {\em Monatsh.\ Math.\ Phys.} {\bf 19} (1908), 221--245.

\bibitem{S67}
M. Schechter, Basic theory of Fredholm operators, {\em Annali della Scuola Normale Superiore di Pisa}, tome 21, no. 2 (1967), 261--280.

\bibitem{Y61}
D.C. Youla, On the factorization of rational matrices, {\em IRE Trans.\ IT-7} (3) (1961), 172--189.

\end{thebibliography}
\end{document}